\documentclass{amsart}
\usepackage{microtype}
\usepackage{mathtools}
\usepackage[hidelinks]{hyperref}
\usepackage{amsthm, amsfonts, amssymb}
\usepackage{tikz-cd}
\usetikzlibrary{arrows.meta, calc}
\usepackage{macros}
\usepackage{url}
\allowdisplaybreaks

\newtheorem{theorem}{Theorem}
\newtheorem{lemma}[theorem]{Lemma}

\newtheorem{corollary}[theorem]{Corollary}
\newtheorem{observation}[theorem]{Observation}
\newtheorem{proposition}[theorem]{Proposition}

\theoremstyle{definition}
\newtheorem{construction}[theorem]{Construction}
\newtheorem{definition}[theorem]{Definition}
\newtheorem{notation}[theorem]{Notation}
\newtheorem{remark}[theorem]{Remark}
\newtheorem{example}[theorem]{Example}

\begin{document}

\title{The Synthetic Sierpi\'nski Cone}

\author[F.\ Bakke]{Fredrik Bakke}
\address{Department of Mathematical Sciences\\ Norwegian University of Science and Technology\\ Trondheim, Norway
}
\email{fredrik.bakke@ntnu.no}

\author[J.\ Sterling]{Jonathan Sterling}
\address{Computer Laboratory\\ University of Cambridge\\ Cambridge, United Kingdom}
\email{js2878@cl.cam.ac.uk}

\author[M.\,D.\ Williams]{Mark Damuni Williams}
\address{School of Computer Science\\ University of Nottingham\\ Nottingham, United Kingdom}
\email{Mark.Williams2@nottingham.ac.uk}

\author[L.\ Ye]{Lingyuan Ye}
\address{Computer Laboratory\\ University of Cambridge\\ Cambridge, United Kingdom}
\email{ly386@cam.ac.uk}

\begin{abstract}
In domains, categories, and toposes, the \emph{Sierpi\'nski cone} construction glues onto a space a universal closed point lying below all the other points. Although this is a lax colimit, it also enjoys a well-known right-handed universal property: the Sierpi\'nski cone classifies partial maps defined on an open subspace. The situation proves more subtle in synthetic models of space based on extending homotopy type theory with an interval, as in several recent approaches to synthetic higher categories and domains: although \emph{globally} it may well be the case that the Sierpi\'nski cone classifies partial maps, this property cannot hold of all parameterised types without degenerating the theory.  On the other hand, there are reflective subuniverses within which the classifying property nonetheless holds.

We show that the largest subuniverse in which the Sierpi\'nski cone classifies partial maps is the accessible localisation at a family of embeddings parameterised in the interval, and this subuniverse is contained within the \emph{Segal} types; this containment is moreover strict in the sense that when the interval is non-trivial, it is not possible for all Segal types to lie in the subuniverse. We finally extend these results from Sierpi\'nski cones to mapping cylinders, providing a new right-handed universal property for the latter.
 \end{abstract}


\maketitle

\section{Introduction}

In category theory and domain theory, there is a remarkable coincidence involving the Sierpi\'nski interval $\II\equiv \{ 0\hookrightarrow 1 \}$: a figure $\II\xrightarrow{\alpha}\mathbb{C}$ represents an arrow $\alpha \colon \alpha(0)\to \alpha(1)$ in $\mathbb{C}$; on the other hand, an arrow $\mathbb{C}\xrightarrow{\chi}\II$ classifies an ``open'' subspace of $\mathbb{C}$ given by the pre-image of $1\bcolon \II$.

The two-handedness of the interval extends to an identification between the \emph{Sierpi\'nski cone} and the \emph{partial map classifier} of any space $X$, which we illuminate below. The Sierpi\'nski cone $X_\bot$ of a space $X$ is the following co-comma square, which we may compute by means of a pushout:
\[
  \begin{tikzcd}
    X\ar[r,equals,nfold]\ar[d] & X \ar[d,open embedding,"\iota_X"]\\
    \One\arrow[r,swap,closed embedding,"\bot_X"]\ar[ur,phantom,"{\Uparrow}\gamma_X"] & X\mathrlap{_\bot}
  \end{tikzcd}
  \qquad
  \begin{tikzcd}
    X\ar[r,closed embedding,"{(0,X)}"]\ar[d] & \II\times X \ar[d,"\gamma_X"description] & X\ar[l,open embedding,swap,"{(1,X)}"]\ar[dl,sloped,swap,open embedding,"\iota_X"] \\
    \One\ar[r,swap,closed embedding,"\bot_X"] & |[elbow se]|X\mathrlap{_\bot}
  \end{tikzcd}
\]

On the other hand, the (open-) partial map classifier $\eta_X\colon X\hookrightarrow \Lift(X)$ is the partial product of $X$ with the open embedding $\{1\}\hookrightarrow \II$ and classifies spans $X\gets U \hookrightarrow Y$
where $U\hookrightarrow \mathbb{Y}$ is an \emph{open} embedding in the sense of being classified by $\{1\}\hookrightarrow\II$ as below:
\[
  \begin{tikzcd}
    |[elbow nw]|U\ar[r]\ar[d,open embedding] & |[elbow nw]|X\ar[d,open embedding,"\eta_X"description]\ar[r] & \{1\}\ar[d,open embedding]\\
    \mathbb{Y}\ar[r,densely dashed,"\exists!"description]
     & \Lift(X) \ar[r] & \II
  \end{tikzcd}
\]

Now, it so happens that we may define a universal comparison map $\sigma_X\colon X_\bot\to \Lift(X)$ using \emph{either} the universal property of the Sierpi\'nski cone or the universal property of the partial map classifier. In the former case, we construct a lax square (left) and in the latter case we observe that $\iota_X\colon X\hookrightarrow X_\bot$ is an open embedding under reasonable assumptions and construct a partial map (right):

\[
  \begin{tikzcd}
    X\ar[r,equals,nfold]\ar[d] & X \ar[d,open embedding,"\eta_X"]\\
    \One\arrow[r,swap,closed embedding,"\bot_X"]\ar[ur,phantom,"{\Uparrow}"] & \Lift(X)
  \end{tikzcd}
  \qquad
  \begin{tikzcd}
    |[elbow nw]|X\ar[r,equals,nfold]\ar[d,open embedding,swap,"\iota_X"] & X\ar[d,open embedding,"\eta_X"]\\
    X_\bot\ar[r,densely dashed,swap,"\exists!"] & \Lift(X)
  \end{tikzcd}
\]

The promised coincidence is that the induced comparison map $X_\bot\to \Lift(X)$ is invertible. This property also extends to models of higher categories such as simplicial sets: when $\Lift(X) \equiv \textstyle\sum_{(i:\Spx{\One})}X^{(i=1)}$ is the open partial map classifier of a simplicial set, it actually happens that $\Lift(X)$ is the directed join of simplicial sets $\Spx{\Zero}\star X \cong X_\bot$.  A similar correspondence is shown by Taylor~\cite{taylor:2000} to hold in \emph{Abstract Stone Duality}, using the fact that functions $f\colon Y\to X_\bot$ can be defined by means of distributive lattice homomorphisms $f^*\colon \II^{X_\bot} \to \II^Y$.

It may therefore be reasonable to say that the correspondence between these two views on partiality is part of the core logic--geometry duality that underpins (higher) category theory, domain theory, topology, \emph{etc.}

\subsection{Subtleties in the synthetic world}

The situation changes considerably when passing to synthetic notions of space---as in synthetic domain theory, synthetic topology, and synthetic (higher) category theory---when expressed in the language of Martin-L\"of type theory or, indeed, homotopy type theory, extended by a generic interval $\II$. The problem is that in dependent type theory, it is unexpectedly strong to assert an axiom like
\begin{quote}
  \em for any type $X$, the comparison map $X_\bot\to \Lift(X)$ is an equivalence,
\end{quote}
because statements in type theory are automatically invariant under change-of-context. Indeed, Pugh and Sterling~\cite{pugh-sterling-2025} noted that under this assumption, we could deduce
\begin{equation}\label{bad-deduction}
  \textstyle\prod_{(i\bcolon\mathbb{I})}
  \mathsf{isEquiv}\bigl\lparen
    (i=1)_\bot
    \xrightarrow{\sigma_{(i=1)}}
    \Lift(i=1)
  \bigr\rparen
  \tag{\textasteriskcentered}
\end{equation}
from which it follows that the boundary inclusion $\{0\}+\{1\}\hookrightarrow \mathbb{I}$ is an equivalence and hence that all synthetic spaces are codiscrete.

When a dependent type theorist working internally to a category $\mathcal{S}$ says \emph{``Let $X$ be a type\ldots''}, what follows is taking place not in $\mathcal{S}$ but in the free extension $\mathcal{S}[X]$ of $\mathcal{S}$ by an object. So an assumed type is, in the dependent type theoretic interpretation, a \emph{generic} type rather than an \emph{arbitrary} type.  In contrast, a practitioner of the Mitchell--B\'enabou language of toposes would explain the same turn of phrase as a conservative use of \emph{prenex polymorphism} over simple types; a schema of this kind can only be instantiated with a global type, and so the bad deduction \eqref{bad-deduction} would not obtain.

Those who do not wish to sacrifice all the advantages of dependent type theory may instead introduce \emph{modalities} that can be used to safely state assumptions that are meant to hold only globally. These modalities do come at a cost, but they have proved indispensable to prior work in the general development of synthetic higher category theory inside homotopy type theory~\cite{11186122}.

Although either suggestion above may be used to faithfully express the partial map classification property of global Sierpi\'nski cones, we
recall Steve Vickers' parable concerning the use of violence in mathematics~\cite{vickers-categories-2008-elephant}:

\begin{quote}
  As a parable, I think of toposes as gorillas (rather tha[n] elephants). At first they look very fierce and hostile, and the locker-room boasting is all tales of how you overpower the creature and take it back to a zoo to live in a cage---if it’s lucky enough not to have been shot first. When it dies you stuff it, mount it in a threatening pose with its teeth bared and display it in a museum to frighten the children. But get to know them in the wild, and gain their trust, then you begin to appreciate their gentleness and can play with them.
\end{quote}

It was in Vickers' immanent spirit that Pugh and Sterling sought an internally viable analysis of the comparison map $\sigma_X\colon X_\bot\to \Lift(X)$ for \emph{generic} $X$ in the dependent type theoretic sense.

\subsection{Sierpi\'nski completeness beyond sets}

Notwithstanding the fact that we cannot expect the comparison map $\sigma_X \colon X_\bot\to \Lift(X)$ to be an equivalence for generic $X$, there are many types $C$ that are nonetheless \emph{(internally) orthogonal} to all $\sigma_X$, in the sense that the restriction map
\[
  C^{\sigma_X} \colon C^{\Lift(X)} \to C^{X_\bot}
\]
is an equivalence. Pugh and Sterling referred to these ``local'' types as \textbf{\emph{Sierpi\'nski complete}}, and sought to show that they formed an accessible reflective subuniverse in the sense of Rijke~\emph{et al.}~\cite{rijke-shulman-spitters:2020} that is moreover closed under the formation of partial map classifiers $\Lift(X)$. In such a subuniverse, one could compute the Sierpi\'nski cone \emph{as} the partial map classifier without resorting to higher inductive types.
Given the overwhelming size of the localising class \[ \mathcal{L}_{\mathit{Sierp}} \equiv \{\sigma_X\colon X_\bot \to \Lift(X)\mid X\ \mathit{type}\}\text{,}\] reflectivity was a tall order, however. One of the results of Pugh and Sterling was, nevertheless, that in \emph{certain} cases it suffices to restrict attention to the ``little'' comparison maps
\[
  \mathcal{L}_{\mathit{Sierp}}^{\mathit{little}} \equiv \{ \sigma_{(i=1)} \colon {(i=1)}_\bot\to \Lift(i=1) \mid i\bcolon\II\}
\]
which are very small indeed.

\begin{quote}
  A type is \textbf{\emph{little-Sierpi\'nski complete}} when it is orthogonal to $\mathcal{L}_{\mathit{Sierp}}^{\mathit{little}}$.
\end{quote}

In particular, Pugh and Sterling proved in the setting of univalent foundations that a \emph{set}\footnote{By \emph{set} we mean a 0-truncated type.} is Sierpi\'nski complete if and only if it is \emph{little} Sierpi\'nski complete, which makes the Sierpi\'nski complete \emph{sets} an accessible reflective subuniverse of the universe of sets. The main question left open in \emph{op.\ cit.}\ was whether this can be extended to untruncated types, rendering the Sierpi\'nski complete \emph{types} an accessible reflective subuniverse; this is indeed important, because the only synthetic $\infty$-categories that are 0-truncated are necessarily posets. We resolve the question in the positive by verifying that \emph{all} little-Sierpi\'nski complete types are Sierpi\'nski complete and therefore that the two conditions are equivalent (Corollary~\ref{cor:sierp=little-sierp}).

\subsection{Segal and based Segal completeness}

Pugh and Sterling's results emerged from an observation concerning the geometry of \emph{simplices} and \emph{inner horns} in the synthetic setting. We first recall the Segal completeness law from Riehl and Shulman~\cite{riehl-shulman:2017},
which specifies the universal property of compositions in synthetic higher categories:

\begin{quote}
  A type is \textbf{\emph{Segal complete}} when it is orthogonal to the inner horn inclusion
  $\mathcal{L}_{\mathit{Segal}} \equiv \{ \Horn \hookrightarrow \Spx{\Two}\}$.
\end{quote}

The observation of Pugh and Sterling~\cite{pugh-sterling-2025} was that the inclusion $\Horn\hookrightarrow\Spx{\Two}$ is, when viewed as an inclusion of families \emph{over} the interval, the sum of all the evident inclusions $(i=1)_\bot \hookrightarrow \II/i$:
\[
  \begin{tikzcd}[cramped]
    \Horn \ar[d,hookrightarrow]\ar[r,"\cong"] & \sum_{(i\bcolon\II)}(i=1)_\bot\ar[d,hookrightarrow]
    \\
    \Spx{\Two}\ar[r,"\cong"'] & \sum_{(i\bcolon\II)} \II/i
  \end{tikzcd}
\]

Based on this observation, Pugh and Sterling proposed a strengthening of the Segal completeness law called \emph{based Segal completeness}:
\begin{quote}
  A type is \textbf{\emph{based Segal complete}} when it is orthogonal to
  $
    \mathcal{L}_{\mathit{Segal}}^{\mathit{based}} \equiv
    \{ (i=1)_\bot \hookrightarrow \II/i\mid i\bcolon\II \}
  $.
\end{quote}

\paragraph{Based Segal \emph{vs.} Segal completeness}

A priori based Segal completeness is evidently stronger than Segal completeness because left classes are closed under sums, but we will show in \S~\ref{sec:separating-based-segal} that under very weak assumptions (which hold in the $\infty$-topos of simplicial spaces), the inclusion is strict:
we show that under reasonable assumptions, each little Sierpi\'nski cone $(i=1)_\bot$ is Segal complete, and if it were also \emph{based} Segal complete, then each inclusion $(i=1)_\bot\hookrightarrow \II/i$ would be an isomorphism, rendering all synthetic spaces codiscrete.

\paragraph{Relation to Sierpi\'nski completeness}

Under appropriate assumptions (such as when $\II$ forms a dominance), the partial order on $\II$ is obtained from the implication order on $\Prop$ by the indicator function $({-}=1)\colon \II\to\Prop$, from which it is easily seen that the slice $\II/i$ is precisely the partial map classifier $\Lift(i=1)$ and the inclusion $(i=1)_\bot \hookrightarrow \II/i$ is the little comparison map $\sigma_{(i=1)}\colon {(i=1)}_\bot\to \Lift(i=1)$. In this case, a type will be ``little Sierpi\'nski'' complete if and only if it is based Segal complete.

\paragraph{Summary}

Combining everything, we therefore obtain reasonable assumptions under which the following implications hold:
\[
  \begin{tikzcd}
    \text{based Segal complete}
      \ar[r,Rightarrow, "\text{strictly}"]
      \ar[d,Leftrightarrow]
    & \text{Segal complete}
    \\
    \text{little-Sierpi\'nski complete}
      \ar[r,Leftrightarrow]
    & \text{Sierpi\'nski complete}
      \ar[u,Rightarrow, "\text{strictly}"']
  \end{tikzcd}
\]

\subsection{Mapping cylinders, type theoretically}

The Sierpi\'nski cone is a special case of a more general co-comma construction that yields the \emph{open mapping cylinder} or \emph{Artin glueing} of a function $f\colon  X\to Y$:
\[
  \begin{tikzcd}
    X\ar[r,equals,nfold]\ar[d,swap,"f"] & X \ar[d,open embedding,"\iota_f"]\\
    Y\arrow[r,swap,closed embedding,"\bot_f"]\ar[ur,phantom,"{\Uparrow}\gamma_f"] & \OpenMCyl_Y(f)
  \end{tikzcd}
\]

We observe in \S~\ref{sec:mapping-cylinders} is that in the synthetic setting, the above is the sum $ \OpenMCyl_Y(f) \cong \textstyle\sum_{y:Y} (\mathsf{fib}_f(y))_\bot$ of the Sierpi\'nski cones of the \emph{fibres} of $f$.
Thus any insight on Sierpi\'nski cones can be straightforwardly transformed into an insight about general mapping cylinders; this allows us, in particular, to characterise the right-handed universal property of mapping cylinders within certain reflective subuniverses in terms of partial map classifiers, adapting the results of Pugh and Sterling concerning the latter. In particular, we show that under suitable assumptions, the mapping cylinder of a function between Sierpi\'nski complete types can be computed as a sum of partial map classifiers (Corollary~\ref{cor:mapping-cylinders-are-logical}).

\subsection{Discussion of related work}

We situate our work within a present-day trend of unifying several classical and emerging areas of research in synthetic mathematics:
\begin{enumerate}
  \item \textbf{Topology}: synthetic domain theory in the tradition of Scott, Hyland, Phoa, Rosolini, and others~\cite{hyland:1991,phoa:1991,fiore-rosolini:1997:cpos,reus-streicher:1999}, and recently taken up by Sterling and Ye~\cite{sterling2025domainsclassifyingtopoi}; synthetic topology in the sense of Le\v{s}nik~\cite{lešnik2021synthetictopologyconstructivemetric}; synthetic computability theory in the sense of Bauer~\cite{bauer:2006}; synthetic Stone duality in the sense of Cherubini \emph{et al.}~\cite{cherubini_et_al:LIPIcs.TYPES.2024.3}; Abstract Stone Duality in the sense of Taylor~\cite{taylor:2002,taylor:2006,taylor:2000}.
  \item \textbf{Geometry}: synthetic differential geometry in the tradition of Kock and Lawvere~\cite{kock:2006} as continued in the homotopical setting by Wellen~\cite{wellen:2017}; synthetic algebraic geometry as introduced by Blechschmidt~\cite{blechschmidt:2017} and recently taken up by the Gothenburg school~\cite{cherubini-coquand-hutzler:2024}.
  \item \textbf{Categories}: synthetic $(\infty,1)$-category theory as initiated by Riehl and Shulman~\cite{riehl-shulman:2017} and continued by Buchholtz, Gratzer, Weinberger, and Bardomiano Mart\'inez~\cite{buchholtz-weinberger:2023,11186122,Bardomiano_Martínez_2025,Bardomiano_Martínez_2025-2}.
\end{enumerate}

We build most directly upon the work of Pugh and Sterling~\cite{pugh-sterling-2025}, with whom we share the goal of assuming axioms only sparingly so that our results may apply \emph{both} to synthetic topology/domain theory \emph{and} to synthetic higher category theory; unlike much of the cited work in synthetic $(\infty,1)$-category theory, our results apply not only in the standard model of simplicial spaces, but also in non-standard models derived from both domain theory and lattice theory. It remains unclear whether our work has any implications for synthetic geometry.

\subsubsection{Synthetic quasi-coherence}

The driving force behind the unification of the three areas described above was Blechschmidt's identification of the duality between affine spaces and quasi-coherent algebras for the generic models of various geometric theories, closely related to the Kock--Lawvere axiom and generalising both Blass's and Phoa's principles~\cite{blass:1986,phoa:1991}. In contrast to the classical work in these areas, present efforts take univalent foundations as a metatheory, which allows for the study of infinite-dimensional spaces and leads to a smoother treatment of even low-dimensional spaces.

\subsection{Motivations and applications}

\subsubsection{Motivations in denotational semantics}

When a denotational semantics of a programming language is given in synthetic domain theory, the partiality needed for forming fixed points is represented using the partial map classifier monad $\Lift(X)$.
In the usual Moggi semantics of partial programming languages, a program $e\colon X\rightharpoonup Y$ is interpreted by a function $e \colon X\to \Lift Y$ and the universal property of the partial map classifier immediately gives a universal reasoning principle for denotational inequality:
\[
  e\leq_{X\rightharpoonup Y} e' \Longleftrightarrow
  \forall x\pdot
  ex{\downarrow}\Rightarrow e'x{\downarrow} \land ex\leq_{Y} e'x
\]

On the other hand, a \emph{lazy} program $f\colon {\sim}X\rightharpoonup Y$ taking a thunk as an argument would correspond to a function
$
  f\colon \Lift{X}\to\Lift Y
$
and in reasoning about such a function, the universal property of the partial map classifier applies only \emph{a priori} on the right. In a denotational semantics based on CPOs, where the partial map classifier \emph{is} the Sierpi\'nski cone, we could check $f\leq_{{\sim}X\rightharpoonup Y} f'$ by cases on whether the input is defined:
\[
  f\leq_{{\sim}X\rightharpoonup Y} f'\Longleftrightarrow
  (f\bot \leq_{\Lift Y} f'\bot)
  \land
  (f\circ\mathsf{return} \leq_{X\rightharpoonup Y} f'\circ\mathsf{return})
\]

The latter reasoning principle holds synthetically if inequality is interpreted by the observational preorder. On the other hand, if inequality is interpreted by \emph{paths} or \emph{directed homotopy}, then the principle is not universally valid in synthetic domain theory. In any case, it is common in denotational semantics to define entire maps $\Lift{X}\to Y$ by cases on whether the input is defined, and this can be done synthetically \emph{only} when $Y$ is Sierpi\'nski complete.

One further observation is that in synthetic domain theory, it is not usually required that paths within a synthetic domain be composable. This property, equivalent to the Segal condition from higher category theory, was called \emph{path transitivity} in the context of synthetic domains by Fiore and Rosolini~\cite{fiore-rosolini:1997:cpos}. Because Sierpi\'nski completeness is strictly stronger, it follows that any notion of synthetic domain in which lazy functions can be defined by cases on their input's definedness must be path transitive.

\subsubsection{Motivations in higher category theory}

Although our original motivations come from denotational semantics, there are also implications for synthetic (higher) category theory. In non-synthetic category theory, the Sierpi\'nski cone construction corresponds to the operation that freely adds an initial object to a category. In the synthetic setting, if you have an $\infty$-category $C$, it is very rare for the Sierpi\'nski cone $C_\bot$ to be an $\infty$-category, but we may use the $\infty$-categorical reflection to obtain an $\infty$-category $R(C_\bot)$ with the correct lax colimiting universal property. There is, however, a priori no explicit description of the reflection $R(C_\bot)$, unlike in the world of ordinary category theory.

Our results have a few implications for the explicit computation of $R(C_\bot)$. First of all, if we restrict to the world of \emph{Sierpi\'nski complete} $\infty$-categories, then $R(C_\bot)$ is equivalent to the partial map classifier $\Lift(C)$ as we would expect from non-synthetic category theory. Secondly, we have shown that it is not possible for every synthetic $\infty$-category to be Sierpi\'nski complete, which means that synthetic $\infty$-categories cannot have a computationally convenient description of free cocompletion by an initial object without further conditions.
\section{Foundational preliminaries}

There are several different metalanguages in use for working synthetically with higher categories and directed spaces. For simplicity's sake, we work in univalent foundations as may be formalised in homotopy type theory~\cite{hottbook}. Rather than globally assuming an interval object satisfying various axioms, we instead state definitions and theorems parameterised in a bounded distributive lattice $\J$ satisfying just the needful assumptions so that they can be instantiated at will without incurring a change of metalanguage.

In contrast to the original work of Riehl and Shulman~\cite{riehl-shulman:2017}, we therefore do not make use of any special types satisfying strong laws of definitional equality, such as extension types. In contrast to the work of Gratzer \emph{et al.}~\cite{11186122}, we do not extend Martin-L\"of type theory with any modalities except briefly in \S~\ref{sec:exponentiability-of-global-points}, nor do we globally assume any simplicially-inspired axioms on the interval (\emph{e.g.}\ that it is a strict linear order, \emph{etc.}). As a result, our work can be readily mechanised in any stock proof assistant that supports intensional type theory---and, indeed, we have mechanised our main result (Theorem~\ref{thm:horrible}) already in the Agda proof assistant.

We recall a few important notions from univalent foundations that we shall need. First of all, we use the terms ``proposition'' and ``set'' in the sense of the HoTT Book~\cite{hottbook}; however, we shall use the phrase ``there exists'' in the truncated sense of traditional mathematics rather than in the propositions-as-types sense of \emph{op.\ cit.} Much of our paper relies on an understanding of localisation in univalent foundations, for which we refer the reader to Rijke \emph{et al.}~\cite{rijke-shulman-spitters:2020,rijke-christensen-scoccola-opie:2020}. We do, however, recall the open and closed modalities below in order to fix notations.

\begin{definition}
  To each proposition $P$ we may associate two lex modalities, as explained by Rijke \emph{et al.}~\cite{rijke-shulman-spitters:2020}:
  \begin{enumerate}
  \item The \emph{open} modality $X \mapsto X^P$.
  \item The \emph{closed} modality $X\mapsto P*X$, where the join $P*X$ is the pushout  of the span $P\gets P\times X\to X$.
  \end{enumerate}
\end{definition}

\begin{definition}\label{def:P-connectedness}
  Let $P$ be a proposition. A type $X$ shall be called \emph{$P$-connected} when either of the following equivalent conditions hold:
  \begin{enumerate}
    \item The function space $X^P$ is contractible.
    \item The join constructor $X\to P*X$ is an equivalence.
  \end{enumerate}
\end{definition}

The equivalence of the two conditions of Definition~\ref{def:P-connectedness} reflects the \emph{complementarity} of the open and closed modalities. For more detail, we refer the reader to Rijke \emph{et al.}~\cite{rijke-shulman-spitters:2020} or even a classical source on topos theory~\cite{johnstone:topos:1977}.

\section{Synthetic topology in a lattice context}

We shall write $\BDL$ for the category of bounded distributive lattices; for a given bounded distributive lattice $\J$, we will write $\ALG{\J} \colondefeq \J/\BDL$ for the category of bounded distributive lattices equipped with a homomorphism from $\J$, which we shall refer to as $\J$-algebras. Of course, we can view $\J$ itself as a generic $\J$-algebra via its identity map. We will write $\mathsf{U}\colon\BDL\to\SET$ and $\mathsf{U}_\J\colon \ALG{\J}\to\SET$ for evident forgetful functors projecting out carrier sets.  We may regard any bounded distributive lattice as a partial order, taking $i\leq j$ to be either $i\land j = i$ or equivalently $i\lor j = j$.

\begin{notation}\label{notation:indicator-functions}
  We will write $\IsT{-},\IsF{-}\colon \mathsf{U}\J\to\Prop$ for the indicator functions of the subsets $\{1\}, \{0\}\subseteq \J$ respectively, so that we have $\IsT{i}\Leftrightarrow (i=1)$ and $\IsF{i}\Leftrightarrow (i=0)$.
\end{notation}

The functions of Notation~\ref{notation:indicator-functions} are monotone and antitone respectively, and so extend to functors of partial orders $\J\to \Prop$ and $\J\Op\to\Prop$; when $\J$ is \emph{local} in a sense that we shall define straightaway, these functions shall in fact become homomorphisms of bounded distributive lattices.

\begin{definition}\label{def:lattice-conditions}
  A bounded distributive lattice $\J$ is called \emph{strict} when $\IsT{-}\colon \J\to\Prop$ preserves the empty join $0$, and \emph{disjunctive} when $\IsT{-}$ preserves binary joins, and \emph{local} when $\IsT{-}$ preserves all finite joins. $\J$ is \emph{conjunctive} when $\IsF{-}$ sends binary meets to joins, or (equivalently) $\J\Op$ is disjunctive. $\J$ is called \emph{conservative} when  $\IsT{-}$ is an embedding or (equivalently) an order-embedding.
\end{definition}

\begin{remark}
  A strict linear order, as in the standard simplicial spaces model of synthetic higher category theory, will necessarily satisfy all the conditions of Definition~\ref{def:lattice-conditions}.
\end{remark}

Recall that the initial bounded distributive is the two-element lattice $\mathbb{B} = \{ 0 < 1 \}$.

\begin{definition}\label{def:quotient-initial}
  A bounded distributive lattice $\J$ is called \emph{quotient-initial} when either of the following equivalent conditions holds:
  \begin{enumerate}
    \item For every $i\bcolon \J$, either $i=0$ or $i=1$.
    \item The universal map $\mathbb{B}\to \J$ is a regular epimorphism.
  \end{enumerate}
\end{definition}

\begin{proof}
  Regular epimorphisms of bounded distributive lattices are precisely the surjective homomorphisms~\cite[Corollary~3.5.3]{borceux:1994:vol2}, so $\mathbb{B}\to\J$ is a regular epimorphism if and only if every element of $\J$ lies in the image of $\mathbb{B}$.
\end{proof}

\begin{observation}\label{obs:strict-plus-quotient-initial}
  If $\J$ is both strict and quotient-initial, then the universal homomorphism $\mathbb{B}\twoheadrightarrow \J$ is an isomorphism and so $\J$ has exactly two elements.
\end{observation}

\begin{proof}
  The forgetful functor $\mathsf{U}\colon \BDL\to\SET$ is monadic and therefore conservative, and so isomorphisms of bounded distributive lattices are precisely the bijective homomorphisms. By strictness of $\J$, the universal map $\mathbb{B}\to \J$ necessarily preserves non-equality, which constructively implies injectivity in this case because $\mathsf{U}\mathbb{B}$ is finite.
\end{proof}

\subsection{Duality between algebras and spaces}

\begin{definition}
  The \emph{spectrum} of a $\J$-algebra $A$ is defined to be the set of $\J$-algebra homomorphisms from $A$ to $\J$:
  \[
    \Spec{\J} A \colondefeq
    \hom_{\ALG{\J}}(A,\J)
  \]

  The spectrum construction evidently gives a functor \[
    \operatorname{Spec}_\J\colon \ALG{\J}^{\mathsf{op}}\to\SET
  \]
  whose action on maps is given by precomposition.
\end{definition}

\begin{definition}
  The \emph{observational $\J$-algebra} of a type $X$ is defined to be the following $X$-fold product of $\J$ with itself,
  \[
    \Opens{\J}X \colondefeq \textstyle\prod_{(x\bcolon X)}\J \equiv \J^X\text{,}
  \]
  computed in $\ALG{\J}$. This construction restricts to a functor
  \[
    \Opens{\J}\colon \SET\to \ALG{\J}\Op
    \text{.}
  \]

  An element of $\Opens{\J}X$ is called an \emph{observation}.
\end{definition}

\begin{proposition}
  The spectrum construction is right adjoint to the observational algebra construction for any $\J\in\BDL$:
  \[
  \begin{tikzcd}
    \ALG{\J}\Op
    \ar[rr, "\Spec{\J}"{name=0}, swap, curve={height=18pt}]
    &&
    \SET
    \ar[ll, "\Opens{\J}"{name=1}, swap, curve={height=18pt}]
    \arrow["\dashv"{anchor=center, rotate=-90, font=\normalsize}, draw=none, from=1, to=0]
    \end{tikzcd}
  \]

  The unit and counit are computed as follows:
  \begin{gather*}
  \begin{aligned}
    &\eta_X \colon X\to \Spec{\J}\Opens{\J}X\\
    &\eta_X(x) \colondefeq \lambda \chi\bcolon \Opens{X}\pdot\chi(x)
  \end{aligned}
  \qquad
  \begin{aligned}
    &\epsilon_A \colon A\to \Opens{\J}\Spec{\J}A\\
    &\epsilon_A(a) \colondefeq \lambda p\bcolon \Spec{\J}A\pdot p(a)
  \end{aligned}
  \end{gather*}
\end{proposition}

Recent work in synthetic (topology, geometry, category theory) has emphasised various \emph{synthetic quasi-coherence} axioms that are closely related to the Kock--Lawvere axiom from synthetic differential geometry and the Phoa principle of synthetic domain theory and synthetic topology; in each case, certain classes of algebras are asserted to lie within the fixed points of the adjunction $\Opens{\J}\dashv \Spec{\J}$. We follow Sterling and Ye~\cite{sterling2025domainsclassifyingtopoi} by instead studying a general notion of quasi-coherence in terms of the adjunction above.\footnote{This is similar to the approach of \emph{abstract Stone duality}~\cite{taylor:2002}, although the adjunction that we are considering is not monadic.}

\begin{definition}
  We shall refer to the fixed points of the adjunction $\Opens{\J}\dashv \Spec{\J}$ as \emph{quasi-coherent $\J$-algebras} and \emph{affine $\J$-spaces} respectively, making the following restricted adjunction an adjoint equivalence by definition:
  \[
  \begin{tikzcd}
    \mathbf{QCoh}_\J\Op
    \ar[rr, "\Spec{\J}"{name=0}, swap, curve={height=18pt}]
    &&
    \mathbf{Aff}_\J
    \ar[ll, "\Opens{\J}"{name=1}, swap, curve={height=18pt}]
    \arrow["\dashv"{anchor=center, rotate=-90, font = \normalsize}, draw=none, from=1, to=0]
    \end{tikzcd}
  \]
\end{definition}

\begin{definition}
  A $\J$-algebra $A$ is called \emph{stably quasi-coherent} when the quotient $A/(a_i=b_i)_{i\leq n}$ of $A$ by any finite set of identifications is quasi-coherent.
\end{definition}

The following characterisation of Phoa's principle~\cite{phoa:1991} is explained by Sterling and Ye~\cite{sterling2025domainsclassifyingtopoi}.

\begin{definition}[{Sterling and Ye~\cite{sterling2025domainsclassifyingtopoi}}]
  $\J$ is said to satisfy \emph{Phoa's principle} when any of the following equivalent conditions holds:
  \begin{enumerate}
  \item The polynomial $\J$-algebra $\J[\mathsf{x}]$ is quasi-coherent.
  \item Every finitely generated free $\J$-algebra is quasi-coherent.
  \item For any $\alpha\colon \J^\J$ we have $\alpha(i) = \alpha(0) \lor (i \land \alpha(1))$.
  \end{enumerate}
\end{definition}

\subsection{Topology and partial map classification}
We shall write $\II$ to denote the ``interval'', which we define to be the underlying set of $\J$ itself.

\subsubsection{Open and closed subspaces}

The ``observational topology'' $\Opens{\J}X \equiv \J^X$ of a space $X$ has $\II^X$ as its underlying set/space. The role of the interval in representing the intrinsic topology of a given space leads to a simple description of \emph{open} and \emph{closed} subspaces.

\begin{definition}[Open/closed embeddings]
  An embedding $U\hookrightarrow X$ is called an \emph{open embedding} when it arises as the pullback of the inclusion $\{1\}\hookrightarrow \II$ along some observation $\varphi_U\in \Opens{\J}{X}$. Dually, a \emph{closed embedding} $K\hookrightarrow X$ is one that arises as a pullback of the inclusion $\{0\}\hookrightarrow \II$.%
  \[
    \begin{tikzcd}
      |[elbow nw]|U \ar[r] \ar[d,open embedding] & \{1\}\ar[d,open embedding]\\
      X \ar[r,densely dashed, "\exists"'] & \II
    \end{tikzcd}
    \quad
    \begin{tikzcd}
      |[elbow nw]|K \ar[r] \ar[d,closed embedding] & \{0\}\ar[d,closed embedding]\\
      X \ar[r,densely dashed, "\exists"'] & \II
    \end{tikzcd}
  \]
\end{definition}

\begin{definition}[Cloven embeddings]
  We have not assumed that open embeddings are \emph{classified} by $\{1\}\hookrightarrow\II$ in the sense that they arise from unique observations. We will speak therefore of \emph{cloven}
  open/closed embeddings $X\hookrightarrow Y$ to refer to open/closed embeddings equipped with a chosen observation that induces them.
\end{definition}

\begin{remark}
  When $\J$ is conservative, all open embeddings are canonically cloven, and when $\J\Op$ is conservative the same holds for closed embeddings.  Hence when $\J$ is stably quasi-coherent, all open and closed embeddings are cloven uniquely.
\end{remark}

\subsubsection{Classification of partial maps}

Next we develop the theory of partial maps relative to $\J$.

\begin{definition}
  An \emph{open-partial map} from $X$ to $Y$ is given by a map into $Y$ from an \emph{open subspace} of $X$, \emph{i.e.}\ a span $Y\gets U\hookrightarrow X$ in which $U\hookrightarrow X$ is an open embedding.
\end{definition}

\begin{definition}
  A \emph{cloven open-partial map} from $X$ to $Y$ is given by an observation $\varphi\in \Opens{\J}{X}$ together with a function $\{X\mid \varphi\}\to Y$ from the induced open subspace as depicted below:
  \[
    \begin{tikzcd}
      Y &
      |[elbow nw]|\{X\mid \varphi\}\ar[r] \ar[l] \ar[d,open embedding]&
      \{1\} \ar[d,open embedding]
      \\
      &
      X\ar[r,"\varphi"'] &
      \II
    \end{tikzcd}
  \]
\end{definition}

Cloven open-partial maps are classified by means of a \emph{polynomial} or \emph{partial product}\footnote{See Dyckhoff and Tholen~\cite{dyckhoff-tholen:1987} and Johnstone~\cite{johnstone:1992,johnstone:1993} for more on partial products.} construction using the generic open embedding $\{1\}\hookrightarrow \II$. In particular, we define
\[
  \Lift(Y) \colondefeq \textstyle\sum_{(i\bcolon\II)}Y^{\IsT{i}}
\]
and observe that the cloven open embedding $\eta_Y\colon Y\hookrightarrow \Lift(Y)$
mapping $y$ to $(1, \mathsf{const}(y))$
classifies cloven open-partial maps into $Y$ as the dashed map $X\to \Lift(Y)$ is induced \emph{universally} in the situation below:
\[
  \begin{tikzcd}
    |[elbow nw]|\{X \mid \varphi\}\ar[r]\ar[d,open embedding] & |[elbow nw]|Y \ar[r] \ar[d,open embedding, "\eta_Y"description]& \{1\}\ar[d,open embedding] \\
    X\ar[r,densely dashed, "\exists!"] \ar[rr,curve={height=18pt},"\varphi"'] & \Lift(Y)\ar[r,"\pi"] & \II
  \end{tikzcd}
\]

When $\J$ is conservative, cleavings of open-partial maps exist canonically and so the open embedding $\eta_Y\colon Y\hookrightarrow \Lift(Y)$ classifies open-partial maps into $Y$. Although conservativity is a very weak assumption, we prefer to work in as much generality as possible when building up the foundations of synthetic topology; the reader who prefers to ignore these subtleties may simply assume that $\J$ is conservative and ignore all discussion of cleaving.

\subsubsection{The dominance property}

Several results depend on the generic open embedding $\{1\}\hookrightarrow \II$ forming a \emph{dominance} in the sense of Rosolini~\cite{rosolini:1986}, which is to say that $\J$ is conservative and open embeddings are closed under composition.

\begin{definition}
  We will say that $\J$ is \emph{dominant} when the open embedding $\{1\}\hookrightarrow\II$ forms a dominance.
\end{definition}

\begin{observation}
 The dual lattice $\J\Op$ is dominant if and only if the closed embedding $\{0\}\hookrightarrow\II$ for $\J$ forms a dominance.
\end{observation}

\begin{proposition}[{Sterling and Ye~\cite{sterling2025domainsclassifyingtopoi}}]
  If $\J$ itself is stably quasi-coherent, then both $\J$ and $\J\Op$ are conservative and dominant.
\end{proposition}

From the dominance property of $\J$, it is a result of Rosoloni~\cite{rosolini:1986} that the open-partial map classifier construction gives a monad.

\subsection{Cubes and simplices}\label{sec:cubes-and-simplices}

Each bounded distributive lattice $\J$ generates a ``geometry'' of cubes, simplices, horns, \emph{etc.}\ that are all obtained by glueing together various $\J$-spectrums.
It is instructive to view all the cubes $\II^n$ as the spaces dual to the \emph{free finitely generated} $\J$-algebras:
\[
  \II^n \cong \Spec{\J} \J[\mathsf{x}_1,\ldots,\mathsf{x}_n]\text{.}
\]

The simplices $\Spx{n} \equiv \{ \alpha \colon n \to \II \mid i_1 \geq \cdots \geq i_n \}$ are the spaces dual to the \emph{finitely presented} $\J$-algebras that freely add to $\J$ a finite descending chain:
\[
  \Spx{n} \cong \Spec{\J} \J[\mathsf{x}_1, \ldots,\mathsf{x}_n] / \mathsf{x}_1\geq \cdots \geq \mathsf{x}_n\text{.}
\]

We can also consider ``slices'' of $\J$, viewed as partial orders. In particular, the slice $\J/i$ has the following spectrum as its underlying space of elements:
\[
  \II/i \equiv \{j\bcolon\II\mid i \geq j\} \cong \Spec{\J} \J[\mathsf{x}]/ i\geq \mathsf{x}\text{.}
\]

Moreover, it is not difficult to see that the triangle $\Spx{\Two}$ is the \emph{sum} of all the slices $\II/i$, as we have
\begin{align*}
  \Spx{\Two} &\equiv \{(i,j)\mid i\geq j\} \cong \textstyle\sum_{(i\bcolon\II)} \II/i\text{.}
\end{align*}

\begin{observation}[{Pugh and Sterling}]\label{obs:slice-of-interval-as-pmc}
  When $\J$ is conservative, the space $\II/i$ is the open-partial map classifier of the proposition $\IsT{i}$, \emph{i.e.}\ we have $\II/i \cong \Lift\IsT{i}$. Hence under this assumption, the triangle is the sum of ``little'' open-partial map classifiers, $\Spx{\Two} \cong \sum_{(i\bcolon\II)} \Lift\IsT{i}$.
\end{observation}

\subsection{The Sierpi\'nski cone}

  The \emph{Sierpi\'nski cone} $X_\bot$ of a type $X$ is computed by Pugh and Sterling~\cite{pugh-sterling-2025} in the style of dependent type theory to be the sum
  \[
    X_\bot \colondefeq
    \textstyle\sum_{(i\bcolon \II)}
    \IsF{i}*X
  \]
  recalling that $P*X$ is the \emph{closed modality} associated to $P$.

As the summing functor $\sum_{\II}\colon \mathbf{Type}/\II\to \mathbf{Type}$ is a left adjoint, it preserves colimits and so the following is a pushout:
\[
  \begin{tikzcd}[column sep=large]
    \sum_{(i\bcolon\II)}
    \IsF{i}\times X
      \ar[r,hookrightarrow,"\sum_{(i\bcolon\II)}\pi_2"]
      \ar[d,"\sum_{(i\bcolon\II)}\pi_1"']
    &
    \sum_{(i\bcolon\II)}X\ar[d]
    \\
    \sum_{(i\bcolon\II)}\IsF{i}\ar[r,hookrightarrow]
    &
    |[elbow se]| X_\bot
  \end{tikzcd}
\]

Eliminating singletons, we therefore obtain the familiar categorical description of the Sierpi\'nski cone:
\[
  \begin{tikzcd}
    X\ar[r,closed embedding,"{(0,X)}"]\ar[d] & \II\times X \ar[d,"\gamma_X"description] & X\ar[l,open embedding,swap,"{(1,X)}"]\ar[dl,sloped,swap,hookrightarrow,"\iota_X"] \\
    \One\ar[r,swap,closed embedding,"\bot_X"] & |[elbow se]|X\mathrlap{_\bot}
  \end{tikzcd}
\]

\begin{observation}\label{lem:X-01-connected-incl-open}
  When $X$ is $\IsT{0}$-connected in the sense of Definition~\ref{def:P-connectedness}, then the following is a pullback square:
  \[
    \begin{tikzcd}[ampersand replacement=\&]
      X
        \ar[r]
        \ar[d, hookrightarrow, "\iota_X"']
      \&
      \{1\}
        \ar[d,open embedding]
      \\
      X_\bot
        \ar[r,"\pi"']
      \&
      \II\text{,}
    \end{tikzcd}
  \]
  where $\pi\colon X_\bot \to \II$ is the map obtained universally from the following square:
  \[
    \begin{tikzcd}[ampersand replacement=\&]
      X\ar[r,closed embedding,"{(0,X)}"]\ar[d] \& \II\times X\ar[d,"\pi_1"]
      \\
      \One\ar[r,closed embedding,"0"'] \& \II
    \end{tikzcd}
  \]
  Therefore, $\iota_X\colon X\hookrightarrow X_\bot$ is open for $\IsT{0}$-connected $X$.
\end{observation}

\begin{proof}
  Recalling the type theoretic description
  \[
    X_\bot \equiv \textstyle\sum_{(i\bcolon\II)}\IsF{i}*X
  \]
  of the Sierpi\'nski cone, the described map $\pi\colon X_\bot \to \II$ is precisely the first projection function, whose fibre over generic $i\bcolon\II$ is $\IsF{i}*X$. We are asked to check that the fibre of $\pi$ at $1\bcolon \II$ is precisely $X$, \emph{i.e.}\ that the canonical map $X \to \IsF{1}*X$ is an equivalence. But this is one of the equivalent formulations of $X$ being $\IsF{1}$-connected or, equivalently, $\IsT{0}$-connected.
\end{proof}

\subsection{The generic inner horn inclusion}

Another important figure is the \emph{generic inner horn} $\Horn$, which represents the shape of a pair of composable arrows. In type theoretic terms, we may compute the generic inner horn following Riehl and Shulman~\cite{riehl-shulman:2017} as the following subspace of the triangle $\Spx{\Two}$:
\[
  \Horn \colondefeq \{ (i \geq j) \bcolon \Spx{\Two} \mid \IsF{j} \lor \IsT{i}\}
\]

From a geometrical point of view, the above amounts to glueing the interval onto itself end-to-end with the images of each copy of the interval forming open and closed subspaces respectively:
\[
  \begin{tikzcd}
    |[elbow nw]|\One
      \ar[r,closed embedding, "0"]
      \ar[d,open embedding, "1"']
    &
    \II
      \ar[d,open embedding, "({-}\geq 0)"]
    \\
    \II
      \ar[r,closed embedding, "(1\geq {-})"']
    &
    |[elbow se]|\Horn
  \end{tikzcd}
\]

We observed in \S~\ref{sec:cubes-and-simplices} that the triangle $\Spx{\Two}$ is the sum of all the slices $\II/i$. A similar observation can be made of the inner horn.

\begin{observation}[{Pugh and Sterling}]\label{obs:inner-horn-as-sum-of-scones}
  The generic inner horn $\Horn$ is the sum of all the little Sierpi\'nski cones $\IsT{i}_\bot$:
  \[
    \Horn \cong \textstyle\sum_{(i\bcolon\II)}\IsT{i}_\bot
  \]
  Hence the generic inner horn inclusion $\Horn\hookrightarrow\Spx{\Two}$ can be re-cast as the sum of all the inclusions $\IsT{i}_\bot\hookrightarrow \II/i$ as depicted below:
  \[
    \begin{tikzcd}[ampersand replacement=\&]
      |[elbow nw]|\IsT{i}_\bot\ar[r,hookrightarrow]\ar[d,hookrightarrow] \& |[elbow nw]|\Horn \ar[d,hookrightarrow]\ar[r,"\cong"description] \& \sum_{(i\bcolon\II)}\IsT{i}_\bot\ar[d,hookrightarrow]
      \\
      |[elbow nw]|\II/i\ar[r,hookrightarrow]\ar[d] \& |[elbow nw]|\Spx{\Two} \ar[d,"\max"description]\ar[r,"\cong"description] \& \sum_{(i\bcolon\II)} \II/i\ar[d,"\pi_1"]
      \\
      \{i\} \ar[r,hookrightarrow] \& \II\ar[r,equals,nfold] \& \II
    \end{tikzcd}
  \]
\end{observation}

\begin{proof}
  Here we use the fact that the closed modality $\IsF{j}*P$ for a proposition $P$ is precisely the disjunction $\IsF{j}\lor P$.
  \begin{align*}
    \Horn &\equiv \{ (i \geq j) \bcolon \Spx{\Two} \mid \IsF{j} \lor \IsT{i}\}\\
    &\cong \textstyle\sum_{(i\bcolon\II)}\{ j\colon\II \mid \IsF{j}\lor \IsT{i}\}\\
    &\cong \textstyle\sum_{(i\bcolon\II)} \sum_{(j\colon\II)} \IsF{j} * \IsT{i}\\
    &\equiv \textstyle\sum_{(i\bcolon\II)} \IsT{i}_\bot\qedhere
  \end{align*}
\end{proof}

\subsection{Open display \& mapping cylinders}\label{sec:mapping-cylinders}

The Sierpi\'nski cone $X_\bot$ is actually a special case of a more general dependently typed construction that replaces $X$ with a \emph{family} of types, which we here call the \emph{open display cylinder}. The latter is novel to this work, but we show that it is equivalent to the classical notion of \emph{open mapping cylinder} under the usual correspondence between maps and families of types.

\begin{definition}
  Let $x\bcolon B\vdash E[x]$ be a family of types; the \emph{open display cylinder} of $E$ over $B$ is defined to be the sum
  \[
    \OpenMCyl_{(x\bcolon B)}E[x]
    \colondefeq
    \textstyle\sum_{(x\bcolon B)}
    E[x]_\bot
  \]
  of all the Sierpi\'nski cones of the fibres of $E$.
\end{definition}

The \emph{open mapping cylinder} of a function $f\colon E\to B$ is defined to be open display cylinder of the family $x\bcolon B\vdash \mathsf{fib}_f(x)$:
\[
  \OpenMCyl_B(f) \colondefeq \OpenMCyl_{(x\bcolon B)}\mathsf{fib}_f(x)\text{.}
\]

Restricting attention to the open mapping cylinder, we recall that the summing functor $\sum_B\colon\mathbf{Type}/B\to\mathbf{Type}$ preserves colimits and so the following is a pushout square:
\[
  \begin{tikzcd}[column sep=large]
    \sum_{(x\bcolon B)}\mathsf{fib}_f(x)
      \ar[rr, closed embedding, "{\sum_{(x\bcolon B)}(0, \mathsf{fib}_f(x))}"]
      \ar[d, "{\sum_{(x\bcolon B)}\star}"']
    &&
    \sum_{(x\bcolon B)}\II\times \mathsf{fib}_f(x)
      \ar[d, "{\sum_{(x\bcolon B)}\gamma_{\mathsf{fib}_f(x)}}"]
    \\
    \sum_{(x\bcolon B)}\One
      \ar[rr, closed embedding, "{\sum_{(x\bcolon B)}\bot_{\mathsf{fib}_f(x)}}"']
    &&
    |[elbow se]| \OpenMCyl_{(x\bcolon B)}\mathsf{fib}_f(x)
  \end{tikzcd}
\]

Adjusting by the canonical equivalence $E\cong \sum_{(x\bcolon B)}\mathsf{fib}_f(x)$ and a few others, we obtain the familiar categorical description of the open mapping cylinder:
\[
  \begin{tikzcd}
    E\ar[d,"f"']\ar[r,closed embedding, "{(0,E)}"]& \II\times E\ar[d,"\gamma_f"description] & E\ar[l,open embedding, "{(1,E)}"']\ar[dl,hookrightarrow,sloped,"\iota_f"']\\
    B\ar[r,closed embedding, "\bot_f"'] & |[elbow se]|\OpenMCyl_B(f)
  \end{tikzcd}
\]

\begin{observation}
  When $f\colon E\to B$ is a $\IsT{0}$-connected map, the embedding $\iota_f\colon E\hookrightarrow \OpenMCyl_B(f)$ is open.
\end{observation}

\begin{proof}
  By Lemma~\ref{lem:X-01-connected-incl-open}, noting that $f$ is $\IsT{0}$-connected if and only if $x\bcolon B\vdash \mathsf{fib}_f(x)$ is a family of $\IsT{0}$-connected types.
\end{proof}

\subsection{The Sierpi\'nski comparison map}\label{sec:sierpinski-comparison-map}

We now describe the comparison map $\sigma_X\colon X_\bot\to \Lift(X)$. This map exists only for \emph{$\IsT{0}$-connected} $X$ in the sense of Definition~\ref{def:P-connectedness}, but we keep in mind that if $\J$ is strict, then every $X$ is $\IsT{0}$-connected. Evidently, both $\IsT{i}$ and $\IsF{i}$ are $\IsT{0}$-connected for any $i\bcolon\II$.

\begin{construction}
  When $X$ is $\IsT{0}$-connected, we may construct the \emph{undefined} partial element $\mathsf{undef}_X : \Lift(X)$ using the universal property of the partial product $\Lift(X)$:
  \[
    \begin{tikzcd}[column sep=large]
      X
      \ar[d,open embedding, "\eta_X"']
      &
      |[elbow nw, elbow ne]|\IsT{0}
      \ar[l,"!"']
      \ar[r, closed embedding]
      \ar[d,open embedding]
      &
      \{1\}
      \ar[d,open embedding]
      \\
      \Lift(X)
      &
      \One
      \ar[l, densely dashed, "\mathsf{undef}_X"]
      \ar[r,closed embedding, "0"']
      &
      \II
    \end{tikzcd}
  \]
\end{construction}

\begin{construction}
  For $\IsT{0}$-connected $X$, we can construct an interpolation \[ \mathsf{glue}_X\colon \II\times X \to \Lift(X)\] between the undefined element the fully defined element as follows:
  \[
    \begin{tikzcd}[column sep=large]
      X
      \ar[d,open embedding, "\eta_X"']
      &
      |[elbow nw, elbow ne]|X
      \ar[l,equals,nfold]
      \ar[r]
      \ar[d,open embedding,"{(1,X)}"description]
      &
      \{1\}
      \ar[d,open embedding]
      \\
      \Lift(X)
      &
      \II\times X
      \ar[l, densely dashed, "\mathsf{glue}_X"]
      \ar[r, "\pi_1"']
      &
      \II
    \end{tikzcd}
  \]
\end{construction}

\begin{construction}
  The comparison map $\sigma_X\colon X_\bot\to\Lift(X)$ for $\IsT{0}$-connected $X$ can be equivalently constructed using either the universal property of $X_\bot$ or the universal property of $\Lift(X)$:
  \[
    \begin{tikzcd}[column sep=large]
      X \ar[r, closed embedding, "{(0,X)}"]\ar[d]
      & \II\times X
        \ar[d, "\mathsf{glue}_X"]
      \\
      \One
        \ar[r, "\mathsf{undef}_X"']
      & \Lift(X)
    \end{tikzcd}
    \quad
    \begin{tikzcd}[column sep=large]
      X
      \ar[d,open embedding, "\eta_X"']
      &
      |[elbow nw, elbow ne]|X
      \ar[l,equals,nfold]
      \ar[r]
      \ar[d,open embedding,"\iota_X"description]
      &
      \{1\}
      \ar[d,open embedding]
      \\
      \Lift(X)
      &
      X_\bot
      \ar[l, densely dashed, "\sigma_X"]
      \ar[r, "\pi"']
      &
      \II
    \end{tikzcd}
  \]
\end{construction}

The following is a direct corollary of Observations~\ref{obs:slice-of-interval-as-pmc} and~\ref{obs:inner-horn-as-sum-of-scones}.

\begin{corollary}[{Pugh and Sterling}]
  When $\J$ is conservative, the generic inner horn inclusion $\Horn\hookrightarrow\Spx{\Two}$ may be identified with the sum of all the ``little'' comparison maps $\sigma_{\IsT{i}}\colon \IsT{i}_\bot\to \Lift\IsT{i}$ in the sense of the diagram below:
  \[
    \begin{tikzcd}
      |[elbow nw]|\IsT{i}_\bot\ar[r,hookrightarrow]\ar[d,hookrightarrow, "\sigma_{\IsT{i}}"'] & |[elbow nw]|\Horn \ar[d,hookrightarrow]\ar[r,"\cong"description] & \sum_{(i\bcolon\II)}\IsT{i}_\bot\ar[d,hookrightarrow, "\sigma_{\IsT{i}}"]
      \\
      |[elbow nw]|\Lift\IsT{i}\ar[r,hookrightarrow]\ar[d] & |[elbow nw]|\Spx{\Two} \ar[d,"\max"description]\ar[r,"\cong"description] & \sum_{(i\bcolon\II)} \Lift\IsT{i}\ar[d,"\pi_1"]
      \\
      \{i\} \ar[r,hookrightarrow] & \II\ar[r,equals,nfold] & \II
    \end{tikzcd}
  \]
\end{corollary}

\subsection{Relative partial map classifiers}

The Sierpi\'nski comparison map from \S~\ref{sec:sierpinski-comparison-map} exhibits an interesting relationship between the Sierpi\'nski cone $X_\bot$ and the cloven open-partial map classifier $\Lift(X)$. Whereas the Sierpi\'nski cone is built by geometrical means (glueing), the partial map classifier is built by \emph{logical} means; we might think of it as the logical counterpart to the Sierpi\'nski cone.
By the same token, we may consider a logical counterpart to the open display/mapping cylinder, namely the \emph{relative partial map classifier} of Awodey~\cite{Awodey2026-zm}.

\begin{definition}
  Let $x\bcolon B\vdash E[x]$ be a family of types; the \emph{relative cloven open-partial map classifier} of $E$ over $B$ is defined to be the sum
  $
    \Lift_{(x\bcolon B)}E[x] \colondefeq
    \textstyle\sum_{(x\bcolon B)}\Lift(E[x])
  $ of the cloven open-partial map classifiers of the components of $E$.

  Similarly, for a function $f\colon E\to B$ we define the relative cloven open-partial map classifier of $f$ to be the sum of the  cloven open-partial map classifier of its fibres:
  \[
    \Lift_{B}(f) \colondefeq
    \Lift_{(x\bcolon B)}\mathsf{fib}_f(x)
  \]
\end{definition}

\begin{definition}\label{def:mapping-cylinder-comparison}
  Let $x\bcolon B\vdash E[x]$ be a family of $\IsT{0}$-connected types; then we have an evident comparison map
  \[
    \sigma_{x\pdot E[x]}\colon \OpenMCyl_{(x\bcolon B)}E[x] \to \Lift_{(x\bcolon B)}E[x]
  \]
  from the open display cylinder to the relative cloven open-partial map classifier defined as the sum of the interior Sierpi\'nski comparison maps,
  $
    \sigma_{(x\bcolon B)\pdot E[x]} \colondefeq
    \textstyle\sum_{(x\bcolon B)}\sigma_{E[x]}
  $.

  Similarly, when $f\colon E\to B$ is a $\IsT{0}$-connected function, we have the comparison map
  $
    \sigma_{f}\colon \OpenMCyl_B(f)\to \Lift_B(f)
  $ defined analogously.
\end{definition}

In more categorical language, the comparison map above arises from the following evident square:
\[
  \begin{tikzcd}
    E
      \ar[r, closed embedding, "{(0,X)}"]
      \ar[d, "f"']
    & \II\times E
      \ar[d, "\mathsf{glue}_f"]
    \\
    B
      \ar[r, "\mathsf{und}_f"']
    & \Lift_{B}(f)
  \end{tikzcd}
\]

\section{Local subuniverses of spaces}

We now recall in short order several classes of maps against which we shall localise.
\begin{align*}
  \mathcal{L}_{\mathit{Segal}} &\colondefeq \{ \Horn \hookrightarrow \Spx{\Two}\}\\
  \mathcal{L}_{\mathit{Segal}}^{\mathit{based}} &\colondefeq \{ \IsT{i}_\bot\hookrightarrow \II/i \mid i\bcolon\II\}\\
  \mathcal{L}_{\mathit{Sierp}} &\colondefeq \{ \sigma_X \colon X_\bot \to \Lift(X) \mid X~\textit{$\IsT{0}$-connected}\}
  \\
  \mathcal{L}_{\mathit{Sierp}}^{\mathit{little}} &\colondefeq \{ \sigma_{\IsT{i}} \colon \IsT{i}_\bot\hookrightarrow \Lift\IsT{i} \mid i\bcolon\II\}
\end{align*}

\begin{definition}
  A type $C$ is called
  \begin{itemize}
  \item \emph{Segal complete} when it is orthogonal to $\mathcal{L}_{\mathit{Segal}}$,
  \item \emph{based Segal complete} when it is orthogonal to $\mathcal{L}_{\mathit{Segal}}^{\mathit{based}}$,
  \item \emph{Sierpi\'nski complete} when it is orthogonal to $\mathcal{L}_{\mathit{Sierp}} $,
  \item and \emph{little Sierpi\'nski} complete when it is orthogonal to $\mathcal{L}_{\mathit{Sierp}}^{\mathit{little}}$.
  \end{itemize}
\end{definition}

\begin{corollary}
  When $\J$ is conservative, a type is based Segal complete if and only if it is little-Sierpi\'nski complete.
\end{corollary}

\begin{proposition}[{Pugh and Sterling~\cite{pugh-sterling-2025}}]\label{prop:interval-segal}
  When $\J$ satisfies Phoa's principle, $\II$ is Segal complete. If $\J$ is furthermore strict and dominant, then $\II$ is \emph{based} Segal complete.
\end{proposition}

\subsection{Segal \texorpdfstring{$\not=$}{!=} based Segal completeness}\label{sec:separating-based-segal}

It is immediate by Observation~\ref{obs:inner-horn-as-sum-of-scones} that based Segal completeness implies Segal completeness, and it was suggested by Ulrik Buchholtz in private correspondence that the converse might hold, \emph{i.e.}\ that every Segal complete type might also be based Segal complete; we now resolve Buchholtz's conjecture by showing that it cannot hold in any non-trivial lattice context satisfying Phoa's principle.

Our strategy to separate Segal and based Segal completeness is to show that the little Sierpi\'nski cones $\IsT{j}_\bot$ are Segal complete under appropriate assumptions, but clearly must not all be based Segal complete, as then each $\IsT{j}_\bot\hookrightarrow \II/j$ would be an equivalence and thus \emph{all} types would be based Segal complete. To carry out our argument, we must first develop a refined notion of \emph{properness}, which is closely related to various ones from synthetic topology.

\NewDocumentCommand\OpenSub{m}{\mathfrak{O}_{#1}}
\NewDocumentCommand\ClosedSub{m}{\mathfrak{K}_{#1}}

\begin{definition}
  Let $\mathfrak{Q}(X)\subseteq \mathbf{Sub}(X)$ be a replete class of subtypes  specified for all types $X$.\footnote{We will at times freely switch between viewing subtypes in terms of their embedding maps and their characteristic predicates. By \emph{replete} we mean stable under equivalence; of course, assuming the univalence axiom, all classes of subtypes are replete.}
  A type $X$ is said to be \emph{$\mathfrak{Q}$-proper} when for any $\phi \in \Prop$ and $\psi\in \mathfrak{Q}(X)$ we may derive the following dual Frobenius law:
  \[
    \phi \lor \forall x\bcolon X\pdot\psi(x) \Longleftrightarrow \forall x\bcolon X\pdot\phi\lor\psi(x)\text{.}
  \]
\end{definition}

The following example is due to Andrej Bauer.

\begin{example}
  Let $\mathbf{Sub}(X)$ be the class of all subtypes of $X$. If every type is $\mathbf{Sub}$-proper, then the law of excluded middle holds.
\end{example}

\begin{lemma}%
  Let $\mathfrak{Q}$ be a replete class of subobjects. Let $f \colon X \to Y$ and $g \colon X \to Z$ be maps such that $Y$ and $Z$ are $\mathfrak{Q}$-proper. Then the pushout $Y+_XZ$ is $\mathfrak{Q}$-proper.
\end{lemma}

\begin{proof}
  We consider the following pushout diagram:
  \[
    \begin{tikzcd}[ampersand replacement=\&]
      X\ar[r,"f"]\ar[d,"g"'] \& Y\ar[d, "\iota_f"]
      \\
      Z\ar[r, "\iota_g"'] \& |[elbow se]|P
    \end{tikzcd}
  \]

  For any $\chi \colon  P \to \Prop$ we obtain the equivalence
  \[ \forall p\bcolon P\pdot \chi(p) \iff  \forall y\bcolon Y\pdot\chi(y) \land \forall z\bcolon Z\pdot\chi(z) \]
  from the $\Prop$-valued universal property of the pushout.
  Fixing $\phi \in \Prop$ and $\psi \in \mathfrak{Q}(P)$ we have
  \begin{align*}
    &\phi \lor  \forall p\bcolon P\pdot \chi(p)
    \\
    &\quad\Leftrightarrow
    \phi \lor \bigl((\forall y\bcolon Y\pdot\chi(\iota_f(y))) \land (\forall z\bcolon Z\pdot\chi(\iota_g(z)))\bigr)
    \\&\quad\Leftrightarrow
    (\phi \lor \forall y\bcolon Y\pdot\chi(\iota_f(y))) \land (\phi \lor \forall z\bcolon Z\pdot\chi(\iota_g(z)))
    \\&\quad\Leftrightarrow
    (\forall y\bcolon Y\pdot\phi \lor \chi(\iota_f(y))) \land (\forall z\bcolon Z\pdot \phi \lor \chi(\iota_g(z)))
    \\&\quad\Leftrightarrow
    \forall p\bcolon P\pdot \phi \lor \chi(p)\text{.}\qedhere
  \end{align*}
\end{proof}

\begin{corollary}\label{lem:horn-is-proper}
  If $\II$ is $\mathfrak{Q}$-proper, then so is $\Horn$.
\end{corollary}

\begin{notation}
  We will write $\ClosedSub{\J}(X)$ for the class of closed predicates on $X$, recalling that  $\psi\colon X\to \Prop$ is closed when it factors through some observation $X\to \II$ by the indicator function $\IsF{-}\colon\II\to\Prop$.
\end{notation}

\begin{definition}
  Let $X$ be a type; the \emph{observational preorder} on $X$ is defined as follows:
  \[
    x,y\bcolon X\mid
    x\sqsubseteq_{\mathsf{obs}}^X y \Leftrightarrow
    \forall \phi \in \Opens{\J}(X)\pdot
    \phi(x) \le \phi(y)
  \]
\end{definition}

\begin{lemma}\label{lem:obs-preorder-with-greatest-element-proper}
  Any type whose observational preorder has a greatest element is $\ClosedSub{\J}$-proper.
\end{lemma}

\begin{proof}
  We fix $\phi\in\Prop$ and $\alpha\colon X\to \II$ to check
  \[
    \phi \lor \forall x\bcolon X\pdot \IsF{\alpha(x)}\Longleftrightarrow
    \forall x\bcolon X\pdot\phi\lor \IsF{\alpha(x)}\text{.}
  \]

  By definition, \emph{every} function is monotone in the observational preorder; the observational preorder on $\II$ is a subrelation of the lattice order on $\II\equiv\mathsf{U}\J$. Hence any function $\alpha\colon X\to \II$ is monotone where $X$ has its observational preorder and $\II$ has the lattice order. Consequently, $\IsF{\alpha(x)}$ is \emph{antitone} in $x\bcolon X$ and thus so is $\phi\lor \IsF{\alpha(x)}$; we may therefore test their universal quantifications at the maximal element of $X$:
  \begin{align*}
    &\phi \lor \forall x\bcolon X\pdot \IsF{\alpha(x)}
    \\
    &\quad\Leftrightarrow \phi \lor \IsF{\alpha(\top_X)}
    \\
    &\quad\Leftrightarrow \forall x\bcolon X\pdot \phi \lor \IsF{\alpha(t)}
    \qedhere
  \end{align*}
\end{proof}

\begin{lemma}\label{lem:simplices-have-top-obs-elements}
  Assuming Phoa's principle, the simplices $\Spx{n}$ all have top elements in their observational preorders.
\end{lemma}

\begin{proof}
  We will show that $(1\ldots)$ is the greatest element of $\Spx{n}$ in its observational preorder. We fix $t\bcolon \Spx{n}$ and $\alpha\colon \Spx{n}\to \II$ to check that $\alpha(t)\leq \alpha(1\ldots)$. By Phoa's principle, $\alpha$ is monotone in the order on $\Spx{n}$ induced by the inclusion $\Spx{n}\hookrightarrow\II^n$, so the result follows.
\end{proof}

\begin{corollary}\label{cor:simplices-are-proper}
  Assuming Phoa's principle, every simplex $\Spx{n}$ is $\ClosedSub{\J}$-proper.
\end{corollary}

\begin{proof}
  By Lemma~\ref{lem:obs-preorder-with-greatest-element-proper} via Lemma~\ref{lem:simplices-have-top-obs-elements}.
\end{proof}

We finally come to our counterexample.

\begin{lemma}\label{lem:brouwerian-counterxample}
  Assume that Phoa's principle holds. Then for each $j\bcolon\II$, the little Sierpi\'nski cone $\IsT{j}_\bot$ is Segal complete.
\end{lemma}

\begin{proof}
  It is not difficult to characterise the type of triangles $\Spx{\Two}\to \IsT{j}_\bot$ in the little Sierpi\'nski cone $\IsT{j}_\bot$. In what follows, we will write $\iota\colon \Horn\hookrightarrow\Spx{\Two}$ for the inner horn inclusion.
  \begin{align*}
    &(\Spx{\Two}\to\IsT{j}_\bot)
    \\
    &\quad\cong
    \bigl(\Spx{\Two}\to\{i\bcolon\II\mid \IsF{i}\lor \IsT{j}\}\bigr)
    \\
    &\quad\cong \{
      \alpha\colon \Spx{\Two}\to \II
      \mid
      \forall t\bcolon \Spx{\Two}\pdot
      \IsF{\alpha(t)} \lor \IsT{j}
    \}
    \\
    &\quad\textit{because $\Spx{\Two}$ is $\ClosedSub{\J}$-proper (Corollary~\ref{cor:simplices-are-proper}):}
    \\
    &\quad\cong \bigl\{
      \alpha\colon \Spx{\Two}\to\II
      \mid
      (\forall t\bcolon \Spx{\Two}\pdot\IsF{\alpha(t)})\lor \IsT{j}
    \bigr\}
    \\
    &\quad\cong \bigl\{
      \alpha\colon \Spx{\Two}\to\II
      \mid
      (\alpha = \mathsf{const}(0))\lor \IsT{j}
    \bigr\}
    \\
    &\quad\textit{because $\II$ is Segal complete (Proposition~\ref{prop:interval-segal}):}
    \\
    &\quad\cong \bigl\{
      \alpha\colon \Spx{\Two}\to\II
      \mid
      (\alpha\circ\iota = \mathsf{const}(0)\circ\iota)\lor \IsT{j}
    \bigr\}
    \\
    &\quad\cong \{
      \alpha\colon \Horn\to\II
      \mid
      (\forall t\bcolon \Horn\pdot\IsF{\alpha(t)})\lor \IsT{j}
    \}
    \\
    &\quad\textit{because $\Horn$ is $\ClosedSub{\J}$-proper (Lemma~\ref{lem:horn-is-proper} and Corollary~\ref{cor:simplices-are-proper}):}
    \\
    &\quad\cong \{
      \alpha\colon \Horn\to\II
      \mid
      \forall t\bcolon\Horn\pdot\IsF{\alpha(t)}\lor \IsT{j}
    \}
    \\
    &\quad\cong
    (\Horn\to \IsT{j}_\bot)\text{.}
    \qedhere
  \end{align*}
\end{proof}

\begin{corollary}\label{cor:conjecture-implies-quotient-initial}
  Assume that Phoa's principle holds and that Segal completeness implies based Segal completeness. Then $\J$ is quotient-initial in the sense of Definition~\ref{def:quotient-initial}.
\end{corollary}

\begin{proof}
  Every slice $\II/j$ is Segal complete because it is a retract of $\II$, which is Segal complete by Phoa's principle; by Lemma~\ref{lem:brouwerian-counterxample}, we know that $\IsT{j}_\bot$ is Segal complete and thus, by assumption, both are based Segal complete. It therefore follows that the unique extension of the identity map $\IsT{j}_\bot\to\IsT{j}_\bot$ along $\IsT{j}_\bot\hookrightarrow \II/j$ provides an inverse to the latter. Considering the pullback of the isomorphism $\IsT{j}_\bot\hookrightarrow \II/j$ along $\{j\}\hookrightarrow \II/j$ we obtain an isomorphism $\IsF{j}\lor\IsT{j}\cong \{j\}\cong\One$ for each $j\bcolon\II$.
\end{proof}

\begin{corollary}
  It cannot be the case that $\J$ is strict, Phoa's principle holds, and Segal completeness implies based Segal completeness.
\end{corollary}

\begin{proof}
  By Corollary~\ref{cor:conjecture-implies-quotient-initial}, $\J$ is quotient-initial under these assumptions; as $\J$ is also assumed strict, the surjection $\epsilon\colon \Two\twoheadrightarrow \II$ is an isomorphism by Observation~\ref{obs:strict-plus-quotient-initial}. This is contradictory with Corollary~7.11 of Sterling and Ye~\cite{sterling2025domainsclassifyingtopoi}, which states that $\II$ is internally connected in the sense of nullifying any set $M$ with decidable equality, such as $M\equiv \Two$.
\end{proof}

The following corollaries are meant to be interpreted externally.

\begin{corollary}
  For the generic strict linear interval in the $\infty$-topos of simplicial spaces, it is not the case that based Segal completeness and Segal completeness coincide.
\end{corollary}

\begin{corollary}
  For the generic flat strict bounded distributive lattice in the $\infty$-topos of Dedekind cubical spaces, it is not the case that based Segal completeness and Segal completeness coincide.
\end{corollary}

\begin{corollary}
  For the generic bounded distributive lattice in its classifying $\infty$-topos, it is not the case that based Segal completeness and Segal completeness coincide.
\end{corollary}

For the last corollary, we need a further auxiliary lemma because the generic bounded distributive lattice is not strict.

\begin{lemma}
  It is not the case that the generic bounded distributive lattice in its classifying $\infty$-topos is quotient-initial.
\end{lemma}

\begin{proof}
  Quotient-initiality is geometric, so if the generic bounded distributive lattice were quotient-initial then \emph{every} bounded distributive lattice would be quotient-initial.
\end{proof}

\subsection{Sierpi\'nski = little-Sierpi\'nski completeness}

We will now show that Sierpi\'nski and little-Sierpi\'nski completeness are equivalent, solving a problem left behind by Pugh and Sterling---who managed it only within the subuniverse of $0$-truncated types. We warm up with a lemma involving certain ``little'' mapping cylinders over the cloven open-partial map classifier.

\begin{lemma}
  Any little-Sierpi\'nski complete type is orthogonal to the following ``little'' open display cylinder comparison map
  \[
    \sigma_{u\pdot\IsT{\pi u}} \colon
    \OpenMCyl_{(u\bcolon\Lift(X))}\IsT{\pi u}
    \hookrightarrow
    \Lift_{(u\bcolon\Lift(X))}\IsT{\pi u}
  \]
  or equivalently the following open mapping cylinder comparison map
  \[
    \sigma_{\eta_X}\colon
    \OpenMCyl_{\Lift(X)}(\eta_X)
    \hookrightarrow
    \Lift_{\Lift(X)}(\eta_X)
  \]
  of the open embedding $\eta_X\colon X\hookrightarrow\Lift(X)$.
\end{lemma}

(We note that $\eta_X\colon X\hookrightarrow \Lift(X)$ is always $\IsT{0}$-connected, because its fibres are propositions of the form $\IsT{i}$.)

\begin{proof}
  The family $u\bcolon\Lift(X)\vdash \IsT{\pi u}$ is equivalently the fibres of the open embedding $\eta_X\colon X\hookrightarrow\Lift(X)$, so the two comparison maps are the same. Unravelling Definition~\ref{def:mapping-cylinder-comparison}, the comparison map is the sum over $u\bcolon\Lift(X)$ of the little Sierpi\'nski comparison maps $\IsT{\pi u}_\bot\hookrightarrow \Lift\IsT{\pi u}$, which are left-orthogonal to any little-Sierpi\'nski complete type.
\end{proof}

\begin{construction}\label{con:eval-square}
  We shall construct the following square whose upper and lower edges are to be thought of as ``evaluation'' maps:
  \[
    \begin{tikzcd}
      \OpenMCyl_{(u\bcolon\Lift(X))}\IsT{\pi u}
        \ar[r,densely dashed, "\bar\epsilon_X"]
        \ar[d,hookrightarrow, "\sigma_{u\pdot\IsT{\pi u}}"']
      &
      X_\bot
        \ar[d, "\sigma_X"]
      \\
      \Lift_{(u\bcolon\Lift(X))}\IsT{\pi u}
        \ar[r, densely dashed, "\epsilon_X"']
      &
      \Lift(X)
    \end{tikzcd}
  \]

  We may use the universal property of the sum to define the two maps fibrewise over $(i,x)\bcolon\Lift(X)$ as follows:
  \[
    \begin{tikzcd}[row sep=large]
      |[elbow nw]|\IsT{i}_\bot
        \ar[d,hookrightarrow, "\sigma_{\IsT{i}}"']
        \ar[r,hookrightarrow]
        \ar[rr, curve={height=-24pt}, "x_\bot"]
      & \OpenMCyl_{(u\bcolon\Lift(X))}\IsT{\pi u}
        \ar[r,densely dashed, "\bar\epsilon_X"description]
        \ar[d,hookrightarrow, "\sigma_{u\pdot\IsT{\pi u}}"description]
      &
      X_\bot
        \ar[d, "\sigma_X"]
      \\
      \Lift\IsT{i}
        \ar[r,hookrightarrow]
        \ar[rr, curve={height=24pt}, "\Lift(x)"']
      & \Lift_{(u\bcolon\Lift(X))}\IsT{\pi u}
        \ar[r, densely dashed, "\epsilon_X"description]
      &
      \Lift(X)
    \end{tikzcd}
  \]

  The outer squares above are evidently the naturality squares for the comparison transformation $\sigma\colon (-)_\bot \to \Lift$ at each $x\colon \IsT{i}\to X$.
\end{construction}

\begin{construction}[Generic points]
  For each $i\bcolon\II$, the open-partial map classifier $\Lift\IsT{i}$ has a ``generic'' element $\mathsf{gen}_i\bcolon \Lift\IsT{i}$,
  \[
    \mathsf{gen}_i \colondefeq (i, \lambda p\bcolon\IsT{i}\pdot p),
  \]
  lying in the image of $i\bcolon \II/i$ under the inclusion $\II/i\hookrightarrow \Lift\IsT{i}$.
\end{construction}

\begin{construction}\label{con:section-of-eval-map}
  Using the generic points of each $\Lift\IsT{\pi u}$, we may exhibit a section $\delta_X\colon \Lift(X)\to  \Lift_{(u\bcolon\Lift(X))}\IsT{\pi u}$ to the evaluation map $\epsilon_X\colon \Lift_{(u\bcolon\Lift(X))}\IsT{\pi u}\to \Lift(X)$, setting
  $
    \delta_X(u) \colondefeq
    (u, \mathsf{gen}_{\pi u})
  $,
  so that we have a section-retraction diagram
  \[
    \begin{tikzcd}
      \Lift(X)
        \ar[r, "\delta_X"]
        \ar[rr,curve={height=18pt},equals,nfold]
      & \Lift_{(u\bcolon\Lift(X))}\IsT{\pi u}
        \ar[r,"\epsilon_X"]
      & \Lift(X)\text{.}
    \end{tikzcd}
  \]

  Note that the other evaluation map $\bar\epsilon\colon \OpenMCyl_{(u\bcolon\Lift(X))}\IsT{\pi u}\to X_\bot$ does \emph{not} have an analogous section.
\end{construction}

The following result is due to Pugh and Sterling, but we provide a new and more streamlined proof.

\begin{theorem}\label{thm:sierpinski-restriction-has-retraction}
  If $C$ is little-Sierpi\'nski complete, then each Sierpi\'nski restriction map
  $
    C^{\sigma_X}\colon C^{\Lift(X)} \to C^{X_\bot}
  $
  has a retraction.
\end{theorem}

\begin{proof}
  We first compose Constructions~\ref{con:eval-square} and~\ref{con:section-of-eval-map}:
  \[
    \begin{tikzcd}[row sep=large]
      & \OpenMCyl_{(u\bcolon\Lift(X))}\IsT{\pi u}
        \ar[r, "\bar\epsilon_X"]
        \ar[d, hookrightarrow, "\sigma_{u\pdot\IsT{\pi u}}"description]
      & X_\bot
        \ar[d, "\sigma_X"]
      \\
      \Lift(X)
        \ar[r,"\delta_X"]
        \ar[rr,curve={height=18pt}, equals, nfold]
      & \Lift_{(u\bcolon\Lift(X))}\IsT{\pi u}
        \ar[r, "\epsilon_X"]
      & \Lift(X)
    \end{tikzcd}
  \]

  Dualising the above with $C$ we have the following:
  \[
    \begin{tikzcd}[row sep=large]
      C^{\Lift(X)}
        \ar[r, "C^{\epsilon_X}"']
        \ar[rr, curve={height=-18pt},equals,nfold]
        \ar[d, "C^{\sigma_X}"']
      &
      C^{\Lift_{(u\bcolon\Lift(X))}\IsT{\pi u}}
        \ar[r, "C^{\delta_X}"']
        \ar[d, "C^{\sigma_{u\pdot\IsT{\pi u}}}"', "\cong"]
      &
      C^{\Lift(X)}
      \\
      C^{X_\bot}
        \ar[r, "C^{\bar\epsilon_X}"']
      & C^{\OpenMCyl_{(u\bcolon\Lift(X))}\IsT{\pi u}}
    \end{tikzcd}
  \]

  We now define a map $H \colon C^{X_\bot}\to C^{\Lift(X)}$ using the assumed inverse of \[ C^{\sigma_{u\pdot\IsT{\pi u}}}\colon C^{\Lift_{(u\bcolon\Lift(X))}\IsT{\pi u}} \to C^{\OpenMCyl_{(u\bcolon\Lift(X))}\IsT{\pi u}}\] as follows:
  \[
    \begin{tikzcd}[row sep=large]
      |[gray]|C^{\Lift(X)}
        \ar[r, gray, "C^{\epsilon_X}"']
        \ar[rr, gray, curve={height=-18pt},equals,nfold]
        \ar[d, gray, "C^{\sigma_X}"']
      &
      C^{\Lift_{(u\bcolon\Lift(X))}\IsT{\pi u}}
        \ar[r, "C^{\delta_X}"']
      &
      C^{\Lift(X)}
      \\
      C^{X_\bot}
        \ar[r, "C^{\bar\epsilon_X}"']
      & C^{\OpenMCyl_{(u\bcolon\Lift(X))}\IsT{\pi u}}
        \ar[u, "(C^{\sigma_{u\pdot\IsT{\pi u}}})^{-1}"description]
    \end{tikzcd}
  \]

  That $H$ is a retraction of $C^{\sigma_X}$ follows immediately from the outer diagram commuting.
\end{proof}

It is considerably more difficult to show that $C^{\sigma_X}\colon C^{\Lift(X)}\to C^{X_\bot}$ has a \emph{section}, and it was here that Pugh and Sterling~\cite{pugh-sterling-2025} resorted to the unpleasant assumption that $C$ is $0$-truncated. We have now succeeded to prove the result unconditionally.

\begin{definition}
  Let $X$ and $C$ be types. A \emph{Sierpi\'nski cone datum} for $X$ in $C$ is given by the following data:
    \begin{align*}
      f^\bot\colon C,\quad
      f^\gamma\colon \II\times X\to C,\quad
      f^H\colon \textstyle\prod_{(x\bcolon X)} f^\bot = f^\gamma(0,x)
    \end{align*}

  We will write $\mathsf{SCD}_X(C)$ for the type of Sierpi\'nski cone data for $X$ in $C$. The evident evaluation map $C^{X_\bot}\to \mathsf{SCD}_X(C)$ is an equivalence by definition.
\end{definition}

\begin{definition}
  A \emph{homotopy} $\phi\colon f\approx_{\mathsf{SCD}_X(C)} g$ of Sierpi\'nski cone data
  is given by the following data:
  \begin{align*}
    \phi^\bot &\colon f^\bot = g^\bot \\
    \phi^\gamma&\colon \textstyle\prod_{(i\bcolon\II,x\bcolon X)} f^\gamma(i,x) = g^\gamma(i,x)\\
    \phi^H&\colon \textstyle\prod_{(x\bcolon X)}
    \left\{
    \begin{tikzcd}[ampersand replacement=\&, cramped]
      f^\bot
        \ar[r, Rightarrow, nfold, "{f^H(x)}"]
        \ar[d, Rightarrow, nfold, "\phi^\bot"']
      \& f^\gamma(0,x)
        \ar[d, Rightarrow, nfold, "{\phi^\gamma(0,x)}"]
      \\
      g^\bot
        \ar[r, Rightarrow, nfold, "{g^H(x)}"']
      \& g^\gamma(0,x)
    \end{tikzcd}
    \right\}
  \end{align*}

  The evident map $f=g \to f\approx_{\mathsf{SCD}_X(C)}g$ is an equivalence by function extensionality.
\end{definition}

\begin{construction}\label{con:scd-restriction}
  Let $\mathbf{X} \equiv (X,\chi)$ be a type equipped with an explicit witness $\chi\colon \mathsf{isContr}\bigl(X^{\IsT{0}}\bigr)$
  that it is $\IsT{0}$-connected.
  Corresponding to the restriction map $C^{\sigma_X}$, we can restrict a function $f\colon \Lift(X)\to C$ to a Sierpi\'nski cone datum $\sigma_{\mathbf{X}}^*f\bcolon \mathsf{SCD}_X(C)$:
  \begin{align*}
    &(\sigma_{\mathbf{X}}^*f)^\bot \colondefeq f(0, \chi.1)\\
    &(\sigma_{\mathbf{X}}^*f)^\gamma(i,x) \colondefeq f(i, \mathsf{const}(x))\\
    &(\sigma_{\mathbf{X}}^*f)^H(x) \colondefeq \mathsf{ap}_{f(0,{-})}~(\chi.2~(\mathsf{const}(x)))
  \end{align*}
\end{construction}

\begin{definition}
  Let $\mathbf{X} \equiv (X,\chi)$ be an explicitly $\IsT{0}$-connected type as in Construction~\ref{con:scd-restriction}. A \emph{Sierpi\'nski extension} of a given $f\bcolon \mathsf{SCD}_X(C)$ is specified by the following data:
  \[
    \hat{f}\colon \Lift(X)\to C,\quad
    \phi_{\hat{f}} \colon \sigma_{\mathbf{X}}^*\hat{f} \approx_{\mathsf{SCD}_X(C)} f
  \]

  We write $\mathsf{SCE}_{\mathbf{X}}(f)$ for the type of Sierpi\'nski extensions of $f$ with respect to the chosen $\IsT{0}$-connectedness structure.
\end{definition}

\begin{definition}
  We say that $C$ has unique Sierpi\'nski extensions if for all such $\mathbf{X}$ and $f\bcolon \mathsf{SCD}_X(C)$, the type $\mathsf{SCE}_{\mathbf{X}}(f)$ is contractible, and unique \emph{little} Sierpi\'nski extensions if for all $i\bcolon \II$, $f\bcolon \mathsf{SCD}_{\IsT{i}}(C)$, the type $\mathsf{SCE}_{\IsT{i}}(f)$ is contractible, where we have equipped $\IsT{i}$ with its canonical $\IsT{0}$-connectedness structure.
\end{definition}

\begin{observation}\label{obs:sierp-ext}
  Let $\mathbf{X} \equiv (X,\chi)$ be an explicitly $\IsT{0}$-connected type. Then the type of sections to the restriction map $C^{\sigma_X}\colon C^{\Lift(X)}\to C^{X_\bot}$ is canonically equivalent to the type
  $
    \textstyle\prod_{(f\bcolon\mathsf{SCD}_X(C))}
    \mathsf{SCE}_{\mathbf{X}}(f)
  $
  of assignments of Sierpi\'nski extensions to all Sierpi\'nski cone data.
  Moreover, $C$ is (little-)Sierpi\'nski complete if and only if $C$ has unique (little) Sierpi\'nski extensions.
\end{observation}

\begin{construction}
  Let $f\bcolon\mathsf{SCD}_X(C)$ be a Sierpi\'nski cone datum for $X$ in $C$. For each $u\bcolon \Lift(X)$ we define
  $
    \bar\epsilon_{X,u}^*f \bcolon \mathsf{SCD}_{\IsT{\pi u}}(C)
  $
  corresponding to the restriction map $C^{\bar\epsilon_X}$ as follows:
  \begin{align*}
    &\bigl(\bar\epsilon_{X,u}^*f\bigr)^\bot \colondefeq f^\bot \\
    &\bigl(\bar\epsilon_{X,(i,x)}^*f\bigr)^\gamma(j,p) \colondefeq f^\gamma(j, x(p))\\
    &\bigl(\bar\epsilon_{X,(i,x)}^*f\bigr)^H(p) \colondefeq f^H(x(p))
  \end{align*}
\end{construction}

\begin{theorem}\label{thm:horrible}
  For little-Sierpi\'nski complete $C$, there exists an assignment
  \[
    \textstyle\prod_{(f\bcolon\mathsf{SCD}_X(C))}
    \mathsf{SCE}_{\mathbf{X}}(f)
  \]
   of Sierpi\'nski extensions.%
\end{theorem}
\begin{proof}
  Let $\mathbf{X} \equiv (X,\chi)$ be an explicitly $\IsT{0}$-connected type, and let $C$ have unique little Sierpi\'nski extensions.
  To each $f\bcolon \mathsf{SCD}_X(C)$, we may assign a Sierpi\'nksi extension $(\hat{f},\phi_{\hat{f}})\bcolon \mathsf{SCE}_{\mathbf{X}}(C)$. We first define $\hat{f}\colon \Lift(X)\to C$ by setting
  $
    \hat{f}(u) \colondefeq
    \overline{\bar\epsilon_{X,u}^*f}(\mathsf{gen}_{\pi u})\text{,}
  $
  where we have written $\overline{\bar\epsilon_{X,u}^*f}$ for the function part of the assumed little Sierpi\'nski extension of $\bar\epsilon_{X,u}^*f\bcolon \mathsf{SCD}_{\IsT{\pi u}}(C)$.

  It now suffices to construct a homotopy
  $
    \phi_{\hat{f}} \colon \sigma_{\mathbf{X}}^*\hat{f}\approx_{\mathsf{SCD}_X(C)}  f
  $, which amounts to constructing the identifications
  \[
    \phi_{\hat{f}}^\bot \colon \hat{f}(0,\chi.1) = f^\bot,
    \quad \phi_{\hat{f}}^\gamma(i,x) \colon \hat{f}(i,\mathsf{const}(x)) = f^\gamma(i,x),
  \]
  together with a commuting square of identifications
  \[
    \begin{tikzcd}[column sep=huge]
      \hat{f}(0,\chi.1)
        \ar[rr, Rightarrow, nfold, "{\mathsf{ap}_{\hat{f}(0,-)}~(\chi.2~(\mathsf{const}(x)))}"]
        \ar[d, Rightarrow, nfold, "\phi_{\hat{f}}^\bot"']
      && \hat{f}(0,\mathsf{const}(x))
        \ar[d, Rightarrow, nfold, "{\phi_{\hat{f}}^\gamma(0,x)}"]
      \\
      f^\bot
        \ar[rr, Rightarrow, nfold, "{f^H(x)}"']
      && f^\gamma(0,x)\text{.}
    \end{tikzcd}
  \]

  Writing $\psi_i : \mathsf{isContr}\bigl(\IsT{0}\to\IsT{i}\bigr)$ for the canonical witness that $\IsT{i}$ is $\IsT{0}$-connected, we define the first identification as follows:
  \[
    \begin{tikzcd}[column sep=huge, row sep=large]
      \hat{f}(0,\chi.1)
        \ar[d, equals, nfold]
        \ar[rr, densely dashed, Rightarrow, nfold, "\phi_{\hat{f}}^\bot"]
      && f^\bot
      \\
      \overline{\bar\epsilon_{X,(0, \chi.1)}^*f}(0,\lambda p\pdot p)
        \ar[rr, Rightarrow, nfold, curve={height=18pt}, "{\mathsf{ap}_{\overline{\bar\epsilon_{X,(0, \chi.1)}^*f}(0, {-})}~(\mathsf{sym}~(\psi_0.2~(\lambda p\pdot p)))}"']
      && \overline{\bar\epsilon_{X,(0, \chi.1)}^*f}(0, \psi_0.1)
        \ar[u, Rightarrow, nfold, "\phi_{\overline{\bar\epsilon_{X,(0, \chi.1)}^*f}}^\bot"']
    \end{tikzcd}
  \]

  Our strategy to witness $\phi^\gamma_{\hat{f}}(i,x)\colon \hat{f}(i,\mathsf{const}(x)) = f^\gamma(i,x)$ is to note that \[ \bar\epsilon_{X,(i,\mathsf{const}(x))}^*f\bcolon \mathsf{SCD}_{\IsT{i}}(C) \] already has a simpler Sierpi\'nski extension, which by our uniqueness assumption may be identified with the assumed one.

  In particular, we may define $h\colon \Lift\IsT{i}\to C$ and a homotopy $\phi_h \colon \sigma_{\IsT{i}}^*h \approx_{\mathsf{SCD}_{\IsT{i}}(C)} \bar\epsilon_{X,(i,\mathsf{const}(x))}^*f$ by setting
  \[
    h_i(u) \colondefeq f^\gamma(\pi u, x),\qquad
    \phi_{h_i}^\bot \colondefeq \mathsf{sym}~(f^H(x))\text{.}
  \]

  The remaining coherences $\phi_h^\gamma,\phi_h^H$ are trivial and their definitions play no role in the rest of the construction. Writing
  \[ \zeta_{\bar\epsilon_{X,(i,\mathsf{const}(x))}^*f}(h, \phi_h)
      \colon
      (h,\phi_h) =
      (\overline{\bar\epsilon_{X,(i,\mathsf{const}(x))}^*f},
        \phi_{\bar\epsilon_{X,(i,\mathsf{const}(x))}^*f})
  \]
  for the resulting identification of extensions, we define
  \[
    \begin{tikzcd}[column sep=huge]
      \hat{f}(i,\mathsf{const}(x))
        \ar[d,equals,nfold]
        \ar[rr,Rightarrow,nfold,densely dashed, "{\phi_{\hat{f}}^\gamma(i,x)}"]
      && f^\gamma(i,x)
      \\
      \overline{\bar\epsilon_{X,(i,\mathsf{const}(x))}^*f}(\mathsf{gen}_i)
        \ar[rr,Rightarrow,nfold,curve={height=18pt},"{\mathsf{ap}_{-.1}~(\zeta_{\bar\epsilon_{X,(i,\mathsf{const}(x))}^*f}(h_i, \phi_{h_i}))}"']
      &&
      h_i(\mathsf{gen}_i)\text{.}
        \ar[u,equals,nfold]
    \end{tikzcd}
  \]

  For the final coherence $\phi_{\hat{f}}^H$, we must fill in the following square:
  \[
    \begin{tikzcd}[row sep=huge]
      \overline{\bar\epsilon_{X,(0,\chi.1)}^*f}(\mathsf{gen}_0)
        \ar[rr, curve={height=-18pt}, Rightarrow, nfold, "{\mathsf{ap}_{\hat{f}(0,-)}~(\chi.2~(\mathsf{const}(x)))}"]
        \ar[d,Rightarrow,nfold,"{\mathsf{ap}_{\overline{\bar\epsilon_{X,(0,\chi.1)}^*f}(0,{-})}~(\mathsf{sym}~(\psi_0.2~(\lambda p\pdot p)))}"description]
      &&\overline{\bar\epsilon_{X,(0,\mathsf{const}(x))}^*f}(\mathsf{gen}_0)
        \ar[dd,Rightarrow,nfold, "{\mathsf{ap}_{-.1}~(\zeta_{\bar\epsilon_{X,(0,\mathsf{const}(x))}^*f}(h, \phi_h))}"description]
      \\
      \overline{\bar\epsilon_{X,(0,\chi.1)}^*f}(0,\psi_0.1)
        \ar[d,Rightarrow,nfold, "\phi_{\overline{\bar\epsilon_{X,(0, \chi.1)}^*f}}^\bot"description]
      \\
      f^\bot
        \ar[rr,Rightarrow,nfold,"f^H(x)"']
      && f^\gamma(0,x)
    \end{tikzcd}
  \]

  We proceed by path induction on \[ \chi.2~(\mathsf{const}(x))\colon \chi.1 = \mathsf{const}(x),\quad \psi_0.2~(\lambda p\pdot p)\colon \psi_0.1 = \lambda p\pdot p\text{,}\] simplifying the problem considerably:
  \[
    \begin{tikzcd}[row sep=huge, column sep=huge]
      \overline{\bar\epsilon_{X,(0,\mathsf{const}(x))}^*f}(\mathsf{gen}_0)
        \ar[d,Rightarrow,nfold, "\phi_{\overline{\bar\epsilon_{X,(0, \mathsf{const}(x))}^*f}}^\bot"']
        \ar[rr,Rightarrow,nfold, "{\mathsf{ap}_{-.1}~(\zeta_{\bar\epsilon_{X,(0,\mathsf{const}(x))}^*f}(h, \phi_h))}"]
      &&f^\gamma(0,x)
      \\
      f^\bot
        \ar[urr, Rightarrow,sloped,"f^H(x)"']
    \end{tikzcd}
  \]

  Then we generalise over the pair
  \[ (\overline{\bar\epsilon_{X,(0,\mathsf{const}(x))}^*f},
  \phi_{\bar\epsilon_{X,(0,\mathsf{const}(x))}^*f})\]  and perform path induction on
  \[ \zeta_{\bar\epsilon_{X,(0,\mathsf{const}(x))}^*f}(h, \phi_h)
      \colon
      (h,\phi_h) =
      (\overline{\bar\epsilon_{X,(0,\mathsf{const}(x))}^*f},
        \phi_{\bar\epsilon_{X,(0,\mathsf{const}(x))}^*f})
  \]
  and our problem reduces as follows:
  \[
    \begin{tikzcd}[row sep=large, column sep=huge]
      f^\gamma(0,x)
        \ar[d,Rightarrow,nfold, "{\mathsf{sym}~(f^H(x))}"']
        \ar[r,equals,nfold]
      &f^\gamma(0,x)
      \\
      f^\bot
        \ar[ur, Rightarrow,sloped,"f^H(x)"']
    \end{tikzcd}
  \]

  This is the symmetry law for path concatenation.
\end{proof}

\begin{corollary}\label{cor:sierpinski-restriction-has-section}
  If $C$ is little-Sierpi\'nski complete, then each Sierpi\'nski restriction map
  $
    C^{\sigma_X}\colon C^{\Lift(X)} \to C^{X_\bot}
  $
  has a section.
\end{corollary}

\begin{proof}
  This is equivalent to the assignment from Theorem~\ref{thm:horrible}, noting Observation~\ref{obs:sierp-ext}.
\end{proof}

\begin{corollary}\label{cor:sierp=little-sierp}
  A type is Sierpi\'nski complete if and only if it is little-Sierpi\'nski complete.
\end{corollary}
\begin{proof}
  By Theorem~\ref{thm:sierpinski-restriction-has-retraction} and Corollary~\ref{cor:sierpinski-restriction-has-section}.
\end{proof}

\begin{corollary}
  The Sierpi\'nski complete types form an accessible reflective localisation.
\end{corollary}

\begin{proof}
  The little-Sierpi\'nski complete types are obviously the reflective localisation at the family $\{ \IsT{i}_\bot\hookrightarrow\Lift\IsT{i} \mid i \bcolon\II\}$, and we have shown in Corollary~\ref{cor:sierp=little-sierp} that these are exactly the Sierpi\'nski complete types.
\end{proof}

\section{Exponentiability in local subuniverses}

We will now generalise the results of Pugh and Sterling~\cite{pugh-sterling-2025} of closure under partial map classifiers $\Lift(X)$ to consider closure under relative open-partial map classifiers $\Lift_{B}(f)$ for functions $f\colon E\to B$.

\begin{definition}
  Let $\mathcal{S}, \mathcal{L}$ be classes of maps. We define $\mathcal{S}^*\mathcal{L}$ to be the collection of functions $f'\colon A' \to B'$ arising as pullbacks of some $f\in \mathcal{L}$ along some $g\in \mathcal{S}$ as depicted below:
  \[
    \begin{tikzcd}
      |[elbow nw]|A'
        \ar[r]
        \ar[d, "f'\in \mathcal{S}^*\mathcal{L}"']
      & A
        \ar[d, "f\in \mathcal{L}"]
      \\
      B'
        \ar[r, "g\in\mathcal{S}"']
      & B
    \end{tikzcd}
  \]

  When $\mathcal{S}$ is the class of \emph{all} maps, we shall write $\Pb{\mathcal{L}}$ for $\mathcal{S}^*\mathcal{L}$.
\end{definition}

\begin{definition}
  A function $f\colon A\to B$ is called \emph{$\mathcal{L}$-exponentiable} when every $\mathcal{L}$-local type is also $\Pb{f}^*\mathcal{L}$-local.
\end{definition}

\begin{notation}
  For any function $p\colon E\to B$, we shall write
  \[
    \mathcal{P}_{p}(X) \colondefeq \textstyle\sum_{(x\bcolon B)}X^{\mathsf{fib}_p(x)}
  \] for the corresponding partial product functor.
\end{notation}

Theorem~\ref{thm:closure-under-generalised-logical-mapping-cylinder} below generalises Fiore's well-completeness lemma as presented in Lemma~4.5 of Pugh and Sterling~\cite{pugh-sterling-2025}.

\begin{theorem}\label{thm:closure-under-generalised-logical-mapping-cylinder}
  Let $\mathcal{L}$ be a class of maps, and let $p\colon E\to B$ be an arbitrary function such that $B$ is $\mathcal{L}$-local. Then the class of $\mathcal{L}$-local types is closed under formation of sums \[
    \textstyle \sum_{(y\bcolon Y)}
    \mathcal{P}_p(\mathsf{fib}_f(y))
  \]
  for $\Pb{p}^*\mathcal{L}$-local maps $f\colon X\to Y$.
\end{theorem}

\subsection{Closure under partial map classifiers}

The following lemmas are proved similarly to Lemma~5.11 and Theorem~5.12 of Pugh and Sterling~\cite{pugh-sterling-2025}.

\begin{lemma}\label{lem:1-exponentiability}
  Suppose that $\J$ is disjunctive and satisfies Phoa's principle. Then $\{1\}\hookrightarrow \II$ is $\mathcal{L}_{\mathit{Segal}}$-exponentiable.
\end{lemma}

\begin{lemma}\label{lem:0-exponentiability}
  Suppose that $\J$ is conjunctive and satisfies Phoa's principle. Then $\{0\}\hookrightarrow \II$ is $\mathcal{L}_{\mathit{Segal}}$-exponentiable.
\end{lemma}

\begin{lemma}
 If $\J$ is local and dominant, then $\{1\}\hookrightarrow\II$ is $\mathcal{L}_{\mathit{Segal}}^{\mathit{based}}$-exponentiable.
\end{lemma}

The following generalises the results of Pugh and Sterling~\cite{pugh-sterling-2025} on closure under open-partial map classifiers.

\begin{theorem}\label{thm:segal-relative-pmc}
  When $\J$ is disjunctive, dominant, and satisfies Phoa's principle, the Segal complete types are closed under relative open-partial map classifiers of functions $p\colon E\to B$ between Segal complete types. If $\J$ is furthermore strict, the based Segal complete or (equivalently) Sierpi\'nski complete types are closed under formation of relative open-partial map classifiers.
\end{theorem}

\begin{proof}
  Let $f\colon E\to B$ be a function between Segal complete types. We apply Theorem~\ref{thm:closure-under-generalised-logical-mapping-cylinder}, and we must check that $f\colon E\to B$ is local with respect to
  $\Pb{\{1\}\hookrightarrow\II}^*\mathcal{L}_{\mathit{Segal}}$,
  which follows immediately from Lemma~\ref{lem:1-exponentiability}, noting that over a Segal complete base, a fibration is inner if and only if the total space is Segal.
\end{proof}

\begin{theorem}
  When $\J$ is local, dominant, and satisfies Phoa's principle, the based Segal complete types are closed under relative open-partial map classifiers of functions $p\colon E\to B$ between based Segal complete types.
\end{theorem}

\begin{proof}
Identical to the proof of Theorem~\ref{thm:segal-relative-pmc}.
\end{proof}

\begin{corollary}\label{cor:mapping-cylinders-are-logical}
  Let $\J$ be local, dominant, and satisfy Phoa's principle. Then within the subuniverse of Sierpi\'nski complete types, open mapping cylinders are computed as relative open-partial map classifiers.
\end{corollary}

The thrust of Corollary~\ref{cor:mapping-cylinders-are-logical} can be viewed in two ways. First, it is saying that within the subuniverse of Sierpi\'nski complete types, the following is a pushout diagram for any $f\colon X\to Y$ with $X$ and $Y$ both Sierpi\'nski complete:
\[
  \begin{tikzcd}
    X
      \ar[r, closed embedding, "{(0,X)}"]
      \ar[d, "f"']
    & \II\times X
      \ar[d, "\mathsf{glue}_f"]
    \\
    Y
      \ar[r, closed embedding, "\mathsf{und}_f"']
    & |[elbow se]| \Lift_{Y}(f)
  \end{tikzcd}
\]

On the other hand, we may also consider it as saying that the open mapping cylinder of $f\colon X\to Y$ in $\Type_{\mathit{Sierp}}\subseteq\Type$ has an unexpected right-handed universal property:
\begin{quote}
  Within $\Type_{\mathit{Sierp}}$, the open mapping cylinder of $f\colon X\to Y$ classifies points $y\bcolon Y$ equipped with open partial elements of $\mathsf{fib}_f(y)$.
\end{quote}

Of course, using Lemma~\ref{lem:0-exponentiability} we may derive a version of the above for \emph{closed} partial map classifiers and and \emph{closed} mapping cylinders under suitably dual hypotheses.

\subsection{Exponentiability of global points}\label{sec:exponentiability-of-global-points}

It is well-known classically that every functor from an $\infty$-category $E$ into $\II$ is exponentiable~\cite[Corollary~2.2.10]{Ayala-Francis:2020}. However, if all point inclusions of $\II$ are $\mathcal{L}_{\mathit{Segal}}$-exponentiable, then Segal completeness implies based Segal completeness, so by Corollary~\ref{cor:conjecture-implies-quotient-initial} the interval would be quotient-initial. How do we square these two facts?

The difference is that constructions made internally to type theory are context-invariant, giving them functorial properties, and this classically-known fact is not sufficiently functorial for the na\"ive interpretation to hold true.

We can express statements of a ``non-functorial'' nature using a modal extension of type theory following Gratzer~\cite{Gratzer_2021}. Concretely, we assume a single mode and a modality $\flat$ equipped with a counit $\varepsilon \colon \flat \to \mathsf{id}$. While this setup is very general, we interpret $\flat$ as ``global points'' or ``underlying core of objects'', since that is the application we have in mind.

\begin{definition}
  We say that $\J$ has \emph{quotient-initial core} if the canonical comparison map $(0,1)\colon \Two \to \Flat\II$ is surjective.
\end{definition}

\begin{proposition}\label{prop:global-point-incl-equivs}
  Let $\J$ be strict with quotient-initial core. Then for every global point inclusion $\{i\} \hookrightarrow \II$ the inclusion $\iota_i\colon \IsT{i}_\bot \hookrightarrow \II/i$ is an equivalence.
\end{proposition}

\begin{proof}
  We need to show $\textstyle\prod_{(i \bcolon \Flat{\II})}\mathsf{isEquiv}(\iota_{\varepsilon(i)})$. Since the core of $\J$ is quotient-initial and $\mathsf{isEquiv}$ is a proposition, it suffices to consider the two cases $i = 0$ and $i = 1$. When $i = 0$, both $(i = 1)_\bot$ and $\II/i$ are contractible (using strictness), and when $i = 1$ the inclusion is the identity function on $\II$.
\end{proof}

\begin{proposition}
  Assume $\J$ is disjunctive and conjunctive, satisfies Phoa's principle, and has quotient-initial core. Then every global point inclusion $\{i\} \hookrightarrow \II$ is $\mathcal{L}_{\mathit{Segal}}$-exponentiable.
\end{proposition}

\begin{proof}
  Follows the same proof outline as Proposition~\ref{prop:global-point-incl-equivs}, using Lemmas~\ref{lem:0-exponentiability} and~\ref{lem:1-exponentiability}.
\end{proof}

\begin{corollary}
  In the classifying $\infty$-topos of strict linear intervals (\emph{i.e.}\ simplicial spaces), every global point inclusion $\{i\} \hookrightarrow \II$ is $\mathcal{L}_{\mathit{Segal}}$-exponentiable.
\end{corollary}
 
\section*{Acknowledgements}
This work was funded in part by the United States Air Force Office of Scientific Research under grant FA9550-23-1-0728 (\emph{New Spaces for Denotational Semantics}; Dr.\ Tristan Nguyen, Program Manager). Views and opinions expressed are however those of the authors only and do not necessarily reflect those of AFOSR.

\bibliographystyle{plain}
\bibliography{refs}

\end{document}